\documentclass[a4paper, reqno, 12pt]{amsart}

\usepackage{geometry}       

\usepackage{float}
\usepackage{bm}
\usepackage{fullpage,xcolor}
\usepackage[mathscr]{euscript}
\usepackage[all]{xy}
\usepackage{epsfig}
\usepackage{amsfonts}
\usepackage{mathptmx}

\usepackage{graphicx}
\usepackage{amssymb} 
\usepackage{amsmath}
\usepackage{amsthm}
\usepackage{mathrsfs}
\usepackage{epstopdf}
\usepackage{url}
\usepackage[msc-links,alphabetic]{amsrefs}
\usepackage{tikz}
\usepackage{color}
\makeatletter

\@addtoreset{equation}{section}
\makeatother 

\theoremstyle{plain}
\newtheorem{theorem}{Theorem}[section]

\newtheorem{proposition}[theorem]{Proposition}
\newtheorem{lemma}[theorem]{Lemma}

\newtheorem{ex}[theorem]{Example}

\theoremstyle{definition}
\newtheorem{definition}[theorem]{Definition}
\newtheorem{example}[theorem]{Example}
\newtheorem{remark}[theorem]{Remark}

\usepackage{latexsym}

\begin{document}

\title[Probabilistic pseudo knot theory]
  {Probabilistic pseudo knot theory}

\author{Ioannis Diamantis}
\address{Department of Data Analytics and Digitalisation,
School of Business and Economics, Maastricht University,
P.O.Box 616, 6200 MD, Maastricht,
The Netherlands.
}
\email{i.diamantis@maastrichtuniversity.nl}

\author{Louis H. Kauffman}
\address{Department of Mathematics, Statistics and Computer Science
University of Illinois at Chicago, 851 South Morgan Street
Chicago, IL, 60607-7045.}
\address{International Institute for Sustainability with Knotted Chiral Meta Matter (WPI-SKCM2),
    Hiroshima University, 1-3-1 Kagamiyama, Higashi-Hiroshima, Hiroshima 739-8526, Japan.}
\email{kauffman@math.uic.edu}
\urladdr{http://www.math.uic.edu/~kauffman/}

\subjclass[2020]{57K10, 57K12, 57K14, 60A99, 57K35, 57K99} 

\keywords{Probabilistic knot theory, pseudo knots, total variation distance, resolution distributions, probabilistic invariants, minimal resolution genus, stochastic topology}

\subjclass[2020]{57K10, 57K12, 57K14, 60A99, 57K35, 57K99} 

\keywords{Probabilistic knot theory, pseudo knots, total variation distance, resolution distributions, probabilistic invariants, minimal resolution genus, stochastic topology}

\setcounter{section}{-1}

\begin{abstract}
We develop the theory of \emph{probabilistic pseudo knots}, providing a framework for modeling knot diagrams with unresolved crossing information. Pseudo knot diagrams generalize classical diagrams by allowing certain crossings to remain unspecified; in the probabilistic setting, each such \emph{pre-crossing}, namely a crossing with undetermined over--under information, is assigned a probability describing the likelihood of resolving as a positive crossing, with complementary probability assigned to the negative resolution. This induces a probability distribution on complete classical resolutions and, by aggregation, a distribution on classical knot types, capturing uncertainty arising in physical, biological, and computational contexts. We introduce \emph{probabilistic equivalence}, defined via total variation distance between resolution distributions, and extend classical numerical quantities such as writhe and linking number to this setting. We also develop new probabilistic constructions, including the probabilistic chirality index, minimal resolution genus, probabilistic Seifert surface distributions, and polynomial invariants extending the Kauffman bracket. We further discuss matrix-based constructions, including probabilistic Seifert and Goeritz-type matrices, as well as probabilistic surgery producing distributions over \(3\)-manifolds. Finally, we discuss potential applications in molecular biology, materials science, and computational topology.
\end{abstract}

%%%%%%%%%%%%%%%%%%%%%%%%

\maketitle

\section{Introduction}\label{intro}

Knot theory has developed into a central area of low-dimensional topology, with deep connections to physics, chemistry, and biology \cites{A,B,F}. Classical knot diagrams, whose crossings carry over--under information, support a rich collection of algebraic, geometric, and quantum invariants used to classify topological entanglement \cites{Rolfsen,CrowellFox}. However, in many experimental or computational settings, including DNA imaging, protein folding, polymer entanglement, and complex woven materials, the precise crossing information may be incomplete, noisy, or
experimentally inaccessible \cites{B,WassermanCozzarelli1986,VirnauMirnyKardar2006,SumnersWhittington1988}. This motivates the development of mathematical frameworks capable of incorporating uncertainty directly into knot diagrams.

A natural step in this direction is the theory of \emph{pseudo knots}. Pseudo knot diagrams were first introduced by Hanaki~\cite{H} as knot diagrams in which some crossings are left unspecified. The theory was subsequently formalized as a topological theory by Henrich, Hoberg, Jablan, Johnson, Minten and Radovi\'c~\cite{HJMR}, through pseudo Reidemeister moves generating pseudo isotopy. These moves coincide locally with the rigid-vertex graph moves for four-valent vertices introduced by the second author in~\cite{Kauffman1989}. Thus a pre-crossing may be viewed either as an unresolved crossing or as a rigid local vertex preserving the cyclic order of the incident strands but carrying no over--under information. Pseudo knots belong to a broader family of diagrammatic extensions of classical knot theory, alongside virtual knot theory~\cites{LK2, D0}. Just as one may consider pseudo versions of classical knots, one may also consider pseudo virtual knots, obtained by allowing pre-crossings in virtual knot diagrams. The probabilistic constructions developed in this paper extend to this virtual setting. They admit diagrammatic and spatial interpretations closely related to singular links and spatial graphs, and have been extended in several directions, including annular and toroidal settings and polynomial invariants for pseudo knots and links~\cites{HD,BJW,D,D1,DLM3,DLM4}.

In this paper we introduce the theory of \emph{probabilistic pseudo knots}. Each pre-crossing is assigned a probability describing the likelihood of resolving as a positive crossing, with complementary probability assigned to the negative resolution. This produces a probability distribution on the set of complete classical resolutions of the diagram,
and also induces a probability distribution on the set of classical knot types arising from these resolutions. The resulting framework combines diagrammatic topology with probabilistic structure and provides a natural setting for comparing uncertain knot diagrams.

To formalize such comparisons, we introduce \emph{probabilistic equivalence}, defined via total variation distance between resolution distributions. Within this framework we extend classical numerical quantities, including the writhe and linking number, by taking suitable expectations over complete resolutions or over the induced resolution distributions. We also introduce structural quantities such as the minimal resolution genus and the
probabilistic chirality index, motivated by the study of molecular chirality and topological chemistry~\cite{F}.

Polynomial and quantum invariants play a central role in modern knot theory. The Kauffman bracket and Jones polynomial \cites{Kauffman,Jones,LK,LK1,WRT}, the HOMFLY-PT polynomial~\cite{HOMFLY}, HOMFLY-PT-type constructions for pseudo links \cite{D1}, and quantum group constructions~\cite{RT} provide powerful tools for distinguishing knots and links. We construct probabilistic analogues obtained by averaging classical diagrammatic or knot-type invariants over the appropriate resolution distributions and analyze their stability under probabilistic equivalence.

We further discuss probabilistic matrix-based constructions, including Seifert matrices~\cite{Seifert}, Goeritz matrices~\cites{Goeritz,GL}, and related determinants and signatures. Although these constructions are algebraic rather than ambient isotopy invariants, they describe how classical structures deform under crossing uncertainty. In addition, motivated by the Lickorish--Wallace theory of surgery \cites{Lickorish,Wallace}, we introduce a framework for \emph{probabilistic surgery}, assigning probability distributions to families of \(3\)-manifolds obtained from classical resolutions.

Beyond their theoretical interest, probabilistic pseudo knots provide a mathematical language for uncertainty in applications where knotting phenomena arise naturally. Potential applications include molecular biology \cites{B,F}, materials science, and computational topology, where
probabilistic sampling, random knot models \cite{SumnersWhittington1988}, and robustness considerations are essential~\cite{EZHLN}.

The paper is organized as follows. In \S~\ref{sec:preliminaries} we introduce probabilistic pseudo knots, their resolution trees, and the resulting spectrum of classical knot types. Section~\ref{sec:probeq} develops probabilistic equivalence via total variation distance and weighted similarity measures. In \S~\ref{distsym} we study probabilistic distances and symmetries, including probabilistic Gordian-type distances, mutation~\cites{Co,Ruberman}, and chirality. Section~\ref{numinv} introduces numerical probabilistic invariants, including probabilistic writhe, linking number, and crossing tendency. In \S~\ref{sec:polyinv} we develop probabilistic polynomial and quantum invariants.
Section~\ref{sec:probSeif} introduces probabilistic Seifert surfaces and minimal resolution genus. Finally, Section~\ref{sec:future} discusses further directions, including probabilistic surgery, additional matrix-based constructions, and potential applications of the theory.

%%%%%%%%%%%%%%%%%%%%%%%%%%%%%%%%%%%%%%%%%%%%%%%%%%%%%%%%%%%%%%%%%%%%%%%%%%%%%%%%%%%%%%%%%%%%%%%%%%%%%%%%%%%%%%%%%%%%%%%%%%%%%%%%%%%%%%%%%%%%%%%%%%%%%%%%%%%%%%%%%%%%%%%%%%%%%%%%%%%%%%%%%%%%%%%%%%%%%%%%%%%%%%%%%%%%%%%%%%%%%%

\section{The theory of probabilistic pseudo knots}\label{sec:preliminaries}

Pseudo knots, introduced by Hanaki~\cite{H}, generalize classical knot diagrams by allowing some crossings to have unspecified over--under information. Such crossings, called \emph{pre-crossings} or \emph{pseudo crossings}, provide a natural diagrammatic model for situations in which crossing information is incomplete, noisy, or experimentally inaccessible. This occurs, for instance, in molecular biology and materials science, where imaging techniques may detect the projection of an entangled object without determining the local height ordering of strands.

We now recall the basic notions of pseudo knot theory that will be needed throughout the paper. Unless otherwise stated, all diagrams are considered in the oriented category.

\subsection{Preliminaries}

\begin{definition}
A \emph{pseudo link diagram} is a generic link diagram in the plane in which some crossings are classical crossings, equipped with over--under information, while others are \emph{pseudo crossings} or \emph{pre-crossings}, where the over--under information is left unspecified (see Figure~\ref{pk1}).
\end{definition}

\begin{figure}[H]
\begin{center}
\includegraphics[width=1.6in]{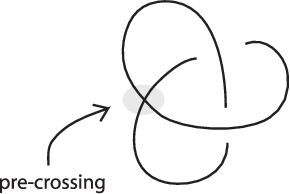}
\end{center}
\caption{A pseudo knot diagram.}
\label{pk1}
\end{figure}

Pseudo link diagrams admit a natural spatial interpretation, introduced in~\cite{DLM3}. Let \(D\) be a planar pseudo link diagram. A \emph{lift} of \(D\) is an embedded object in three-dimensional space constructed as follows:

\begin{itemize}
\item each classical crossing is realized inside a sufficiently small \(3\)-ball as a standard over--under crossing of embedded arcs;

\item each pseudo crossing is supported by a sufficiently small rigid embedded disc whose boundary contains the incident strands and records the absence of prescribed over--under information;

\item the arcs connecting the local crossing models are embedded smoothly in space so as to connect the corresponding local structures.
\end{itemize}

The resulting embedded object is called a \emph{spatial pseudo link} (see Figure~\ref{lif}).

\begin{figure}[H]
\begin{center}
\includegraphics[width=2.4in]{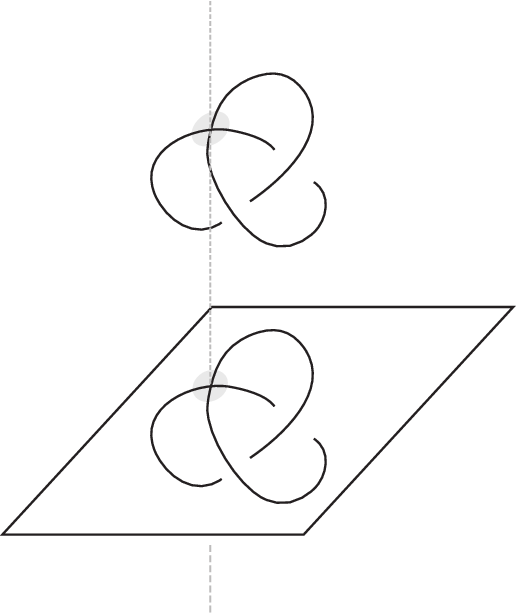}
\end{center}
\caption{The lift of a pseudo knot diagram to a spatial pseudo knot.}
\label{lif}
\end{figure}

Conversely, a spatial pseudo link admits a regular projection onto a plane in which classical crossings project to classical crossings, while the discs supporting pseudo crossings project to pseudo crossings. The resulting projection is a planar pseudo link diagram.

Two oriented spatial pseudo links are said to be \emph{ambient isotopic} if they are related by an ambient isotopy of \(\mathbb{R}^3\) preserving the local ball and disc structures supporting classical and pseudo crossings, respectively; see~\cite{DLM3}. The corresponding Reidemeister theorem for spatial pseudo links can then be stated as follows.

\begin{theorem}[Reidemeister theorem for pseudo links]
Two oriented spatial pseudo links are ambient isotopic if and only if any two planar pseudo link diagrams representing them are related by planar isotopy, the classical Reidemeister moves, and the pseudo Reidemeister moves illustrated in Figure~\ref{reid}.
\end{theorem}

\begin{figure}[H]
\begin{center}
\includegraphics[width=5.8in]{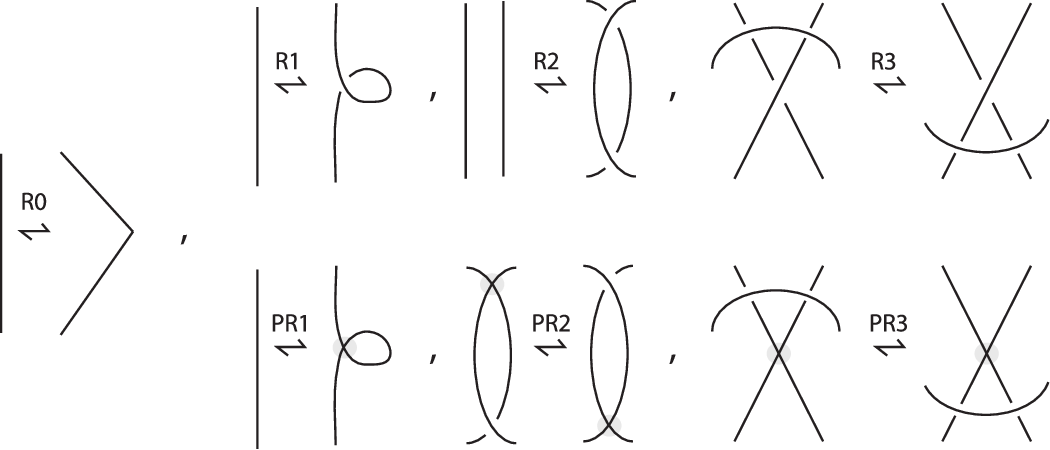}
\end{center}
\caption{Reidemeister moves for pseudo knots.}
\label{reid}
\end{figure}

Finally, following~\cite{DLM4}, regular isotopy for pseudo links extends the notion introduced by Kauffman~\cite{Kauffman}.

\begin{definition}\label{def:pseudo-regular}
\emph{Regular isotopy for pseudo links} is the equivalence relation among pseudo link diagrams generated by planar isotopy \((R0)\), the classical Reidemeister moves \(R2\) and \(R3\), and the pseudo Reidemeister moves \(PR2\) and \(PR3\) illustrated in Figure~\ref{reid}. In particular, the moves \(R1\) and \(PR1\) are excluded.
\end{definition}

\subsection{Probabilistic pseudo knots}

We now introduce probabilistic pseudo knots and links. The basic idea is to enrich a pseudo link diagram by assigning to each pre-crossing a probability describing the likelihood of resolving that pre-crossing as positive. Since the notions of positive and negative crossings require orientations, all probabilistic pseudo knots and links in this paper are considered in the oriented category.

\begin{definition}\label{ppk}
An \emph{oriented probabilistic pseudo link diagram} is an oriented pseudo link diagram in which each pre-crossing \(pc_i\) is decorated with a number \(p_i\in [0,1]\), interpreted as the probability that \(pc_i\) resolves as a positive crossing. That is,
\[
p_i := \mathbb{P}(pc_i=\text{positive crossing}),
\qquad
1-p_i := \mathbb{P}(pc_i=\text{negative crossing}).
\]
When the underlying pseudo link has one component, we call the diagram an \emph{oriented probabilistic pseudo knot diagram}.

If every crossing of a diagram is a pre-crossing equipped with such a probability, the diagram is called a \emph{completely probabilistic pseudo knot} or \emph{completely probabilistic pseudo link}, according to the number of components.
\end{definition}

\begin{remark}\rm
A probabilistic pseudo knot may be viewed both as a labelled diagram and as a probabilistic object. Let \(K\) be a probabilistic pseudo knot with pre-crossings
\(pc_1,\ldots,pc_n\). The set of complete sign assignments is \(\Omega := \{+,-\}^{n}\), where an element of \(\Omega\) records, for each pre-crossing, whether it is resolved as a positive or negative crossing.

Unless stated otherwise, we equip \(\Omega\) with the product probability measure determined by independent Bernoulli random variables
\[
X_i:\Omega\to\{+,-\},
\]
satisfying
\[
\mathbb{P}(X_i=+) = p_i,
\qquad
\mathbb{P}(X_i=-) = 1-p_i.
\]
Thus distinct pre-crossings are assumed to resolve independently. This independence assumption is part of the basic model; more general models with dependent resolutions may be considered separately.
\end{remark}

Classical knot theory is recovered as a degenerate case of this construction. A classical positive crossing may be encoded by a probability label \(p=1\), while a classical negative crossing may be encoded by a probability label \(p=0\) (see Figure~\ref{probcr}). In this sense, a classical diagram determines a degenerate probabilistic pseudo diagram.

\begin{figure}[H]
\begin{center}
\includegraphics[width=2.8in]{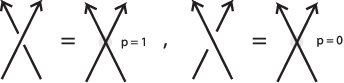}
\end{center}
\caption{Classical crossings encoded as degenerate probabilistic crossings with \(p=1\) or \(p=0\).}
\label{probcr}
\end{figure}

When considering equivalence classes of oriented probabilistic pseudo link diagrams, the probability labels are treated as part of the diagrammatic data. Thus, under isotopy, each labelled pre-crossing is transported together with its probability label. In particular, pseudo Reidemeister moves do not impose algebraic relations among probability labels; they only move or rearrange the labelled pre-crossings already present in the diagram.

\begin{definition}\label{def:prob-pseudo-regular}
\emph{Regular isotopy for probabilistic pseudo links} is the equivalence relation among oriented probabilistic pseudo link diagrams generated by planar isotopy \((R0)\), the classical Reidemeister moves \(R2\) and \(R3\), and the pseudo Reidemeister moves \(PR2\) and \(PR3\), with probability labels on pre-crossings transported along the moves. As in the case of regular isotopy of classical knots, the moves \(R1\) and \(PR1\) are excluded.
\end{definition}

Figure~\ref{orr} illustrates representative oriented Reidemeister and pseudo Reidemeister moves for probabilistic pseudo knots. All moves are understood to occur between oriented strands, allowing all compatible orientations, and all probability labels are carried by the corresponding pre-crossings throughout the move. For detailed discussions of oriented Reidemeister moves and generating sets, see~\cites{NOS,BEHY}.

\begin{figure}[H]
\begin{center}
\includegraphics[width=4.5in]{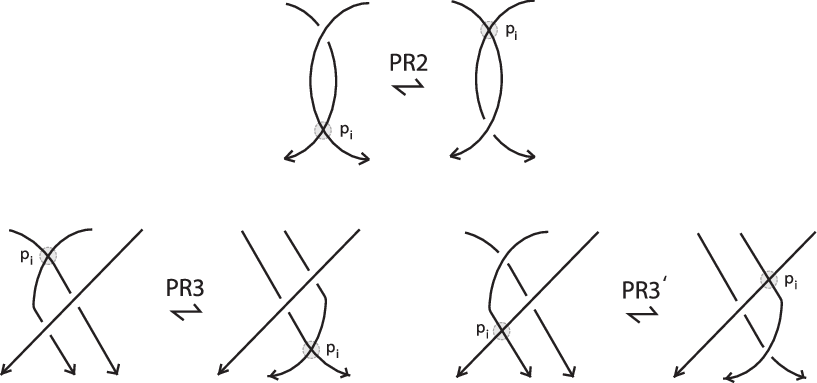}
\end{center}
\caption{Some Reidemeister moves for oriented probabilistic pseudo knots.}
\label{orr}
\end{figure}

%%%%%%%%%%%%%%%%%%%%%%%%%%%%%%%%%%%%%%%%%%%%%%%%%%%%%%%%%%%%%%%%%%%%%%%%%

\subsection{The Probabilistic Resolution Tree}

The probabilistic structure of probabilistic pseudo knots naturally leads to the notion of a {\it probabilistic resolution tree}, a tool that systematically organizes all classical resolutions of a probabilistic pseudo knot together with their associated probabilities. This construction extends the {\it weighted resolution set} introduced in \cite{HJMR} and provides a convenient framework for describing the distribution of classical outcomes arising from probabilistic resolutions. More precisely, we have:

\begin{definition}\rm
Let $K$ be an oriented probabilistic pseudo knot with pre-crossings $pc_1,\ldots,pc_n$, each labeled with a probability $p_i \in [0,1]$. The \emph{probabilistic resolution tree} of $K$, denoted $T(K)$, is a binary tree defined up to isomorphism as follows:

\begin{enumerate}

\item \textbf{Structure.}
Each node represents a pseudo knot diagram obtained by resolving a subset of the pre-crossings of $K$.

\begin{itemize}
\item The \emph{root node} corresponds to $K$ with all pre-crossings unresolved.

\item \emph{Internal nodes} correspond to partial resolutions.

\item \emph{Leaf nodes} correspond to complete resolutions and hence to classical knot diagrams.
\end{itemize}

\smallbreak

\item \textbf{Branches.}
Each edge corresponds to resolving one previously unresolved pre-crossing.

\begin{itemize}
\item A \emph{positive branch} resolves the chosen pre-crossing as a positive crossing and is assigned weight $p_i$.

\item A \emph{negative branch} resolves it as a negative crossing and is assigned weight $1-p_i$.
\end{itemize}

The tree is considered up to the natural isomorphisms induced by reordering the pre-crossings, since different orders of resolving pre-crossings yield the same collection of complete resolutions with the same associated probabilities.

\smallbreak

\item \textbf{Resolution Distribution.}
Each leaf node $C$ determines a classical resolution of $K$. Assuming independence of resolutions at distinct pre-crossings, the probability assigned to $C$ is

\[
\mathbb{P}(C)
=
\prod_{i \in P(C)} p_i
\prod_{j \in N(C)} (1-p_j),
\]

where $P(C)$ (resp.~$N(C)$) denotes the set of pre-crossings resolved positively (resp.~negatively) in $C$.

This assignment defines a probability distribution on the set
\[
\mathcal{C}(K)
:=
\{ \text{all complete classical resolutions of } K \},
\]
since
\[
\sum_{C \in \mathcal{C}(K)} \mathbb{P}(C) = 1 .
\]

\smallbreak

\item \textbf{Dominant and Rare Resolutions.}
Let $\tau \in [0,1]$ be a fixed threshold.

A resolution $C$ is called

\begin{itemize}
\item \emph{dominant} if $\mathbb{P}(C)>\tau$,

\item \emph{rare} if $\mathbb{P}(C)\leq \tau$.
\end{itemize}

The corresponding sets are denoted

\[
\mathcal{C}_{\mathrm{dom}}(K)
\quad \text{and} \quad
\mathcal{C}_{\mathrm{rare}}(K)
=
\mathcal{C}(K)\setminus
\mathcal{C}_{\mathrm{dom}}(K).
\]

\end{enumerate}
\end{definition}

\begin{remark}\rm
Each pre-crossing $pc_i$ may be viewed as an independent Bernoulli random variable taking values in $\{+,-\}$ with probabilities $p_i$ and $1-p_i$, respectively. Thus the probabilistic resolution tree encodes the product probability measure on the set of classical resolutions. Distinct complete resolutions may yield classical diagrams representing the same knot type. Consequently, the probability distribution on $\mathcal{C}(K)$ induces an aggregated probability distribution on classical knot types. The threshold $\tau$ distinguishing dominant and rare resolutions is an auxiliary, user-chosen parameter intended for heuristic or application-driven analysis. It is diagram-dependent and does not define a topological or probabilistic invariant, nor does it play a role in the invariant theory developed later in this paper.
\end{remark}

\begin{ex}\rm
Consider a completely probabilistic trefoil knot $K$ with three pre-crossings labeled $p_1=0.7$, $p_2=0.6$, and $p_3=0.8$. The probabilistic resolution tree $T(K)$ is illustrated in Figure~\ref{ptr}.

\begin{enumerate}

\item \textbf{Root Node.}
Represents $K$ with no pre-crossings resolved.

\smallbreak

\item \textbf{Branches.}
At each stage one pre-crossing is resolved, producing a positive branch with weight $p_i$ and a negative branch with weight $1-p_i$.

\smallbreak

\item \textbf{Leaf Nodes.}
Each leaf corresponds to a classical resolution.

\smallbreak

\item \textbf{Resulting Knot Types.}
In this example three knot types arise:

\begin{itemize}

\item Left-handed trefoil:
\[
P_{LH} = p_1 p_2 p_3 .
\]

\item Right-handed trefoil:
\[
P_{RH} =
\prod_{i=1}^{3} (1-p_i).
\]

\item Unknot:
arising from the remaining resolutions, with probability
\[
P_U = 1- P_{LH}- P_{RH}.
\]

\end{itemize}

\smallbreak

\item \textbf{Dominant and Rare Resolutions.}
For example, if $\tau=0.1$:

\begin{itemize}

\item All positive crossings:
\[
0.7 \cdot 0.6 \cdot 0.8 = 0.336
\]
(dominant).

\item All negative crossings:
\[
(1-0.7)(1-0.6)(1-0.8)=0.024
\]
(rare).

\end{itemize}

\end{enumerate}
\end{ex} 

\begin{figure}
\begin{center}
\includegraphics[width=5.2in]{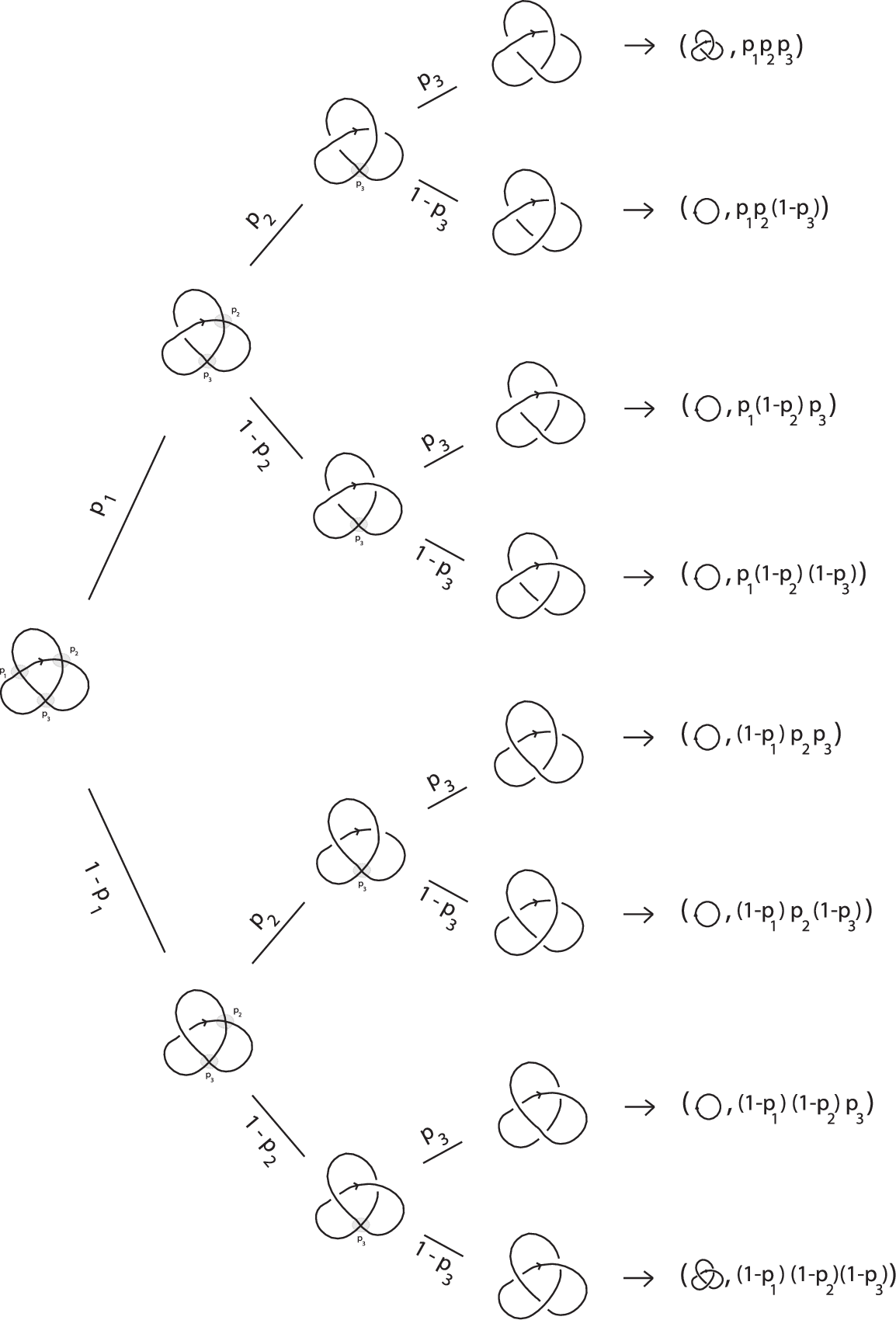}
\end{center}
\caption{Probabilistic resolution tree of the completely probabilistic trefoil knot.}
\label{ptr}
\end{figure}

%%%%%%%%%%%%%%%%%%%%%%%%%%%%%%%%%%%%%%%%%%%%%%%%%%%%%%%%%%%%%%%%%%%%%%%%%%

\subsection{The Spectrum of a Probabilistic Pseudo Knot}

We now introduce a fundamental invariant that captures the full probability distribution over classical knot types obtained from resolving a probabilistic pseudo knot.

\begin{definition}\label{def:spectrum}
Let $P$ be an oriented probabilistic pseudo knot diagram with pre-crossings $c_1,\dots,c_n$, each assigned a probability $p_i\in[0,1]$. Let
\[
\mathcal K(P)
:=
\{\text{oriented classical knot types arising among complete resolutions of }P\}.
\]

The \emph{spectrum} of $P$, denoted $\mathrm{Spec}(P)$, is the probability distribution on $\mathcal K(P)$ defined by

\[
\mathrm{Spec}(P)
=
\left\{
(K,\mathbb P_P[K])
\ \middle|\
K\in\mathcal K(P),\;
\mathbb P_P[K]>0
\right\},
\]

where

\[
\mathbb P_P[K]
=
\sum_{\substack{r\in\mathcal C(P)\\
\mathrm{type}(r)=K}}
\prod_{i\in P(r)} p_i
\prod_{j\in N(r)}(1-p_j).
\]

Here $\mathcal C(P)$ denotes the set of complete resolutions of $P$, while $P(r)$ and $N(r)$ denote the sets of pre-crossings resolved positively and negatively in $r$, respectively.
\end{definition}

\noindent In the special case where all $p_i=\frac12$, each of the $2^n$ resolutions occurs with equal probability. In this case
\[
\mathbb P_P[K]
=
\frac{
\#\{r\in\mathcal C(P)\mid \mathrm{type}(r)=K\}
}{2^n},
\]

so that the spectrum recovers the weighted resolution set introduced in \cite{HJMR}.

\begin{remark}\rm
While the probabilistic resolution tree records all sign assignments and resulting diagrams, the spectrum aggregates this information into a probability distribution on isotopy classes of classical knots. It therefore captures the probabilistic topological behaviour of a probabilistic pseudo knot.
\end{remark}

%%%%%%%%%%%%%%%%%%%%%%%%%%%%%%%%%%

\begin{example}
Suppose $P$ has two pre-crossings $c_1,c_2$, each with probability $\frac12$, and assume that one resolution yields a figure--eight knot while the remaining three yield the unknot. Then

\[
\mathbb P_P[\text{figure--eight}]=\tfrac14,
\qquad
\mathbb P_P[\text{unknot}]=\tfrac34,
\]

and

\[
\mathrm{Spec}(P)
=
\left\{
(\text{figure-eight},\tfrac14),
(\text{unknot},\tfrac34)
\right\}.
\]
\end{example}

\begin{remark}\rm
The spectrum allows the computation of topological probabilities, for example the probability that $P$ resolves into a nontrivial knot.
\end{remark}

\begin{proposition}
Let $P$ and $P'$ be oriented probabilistic pseudo knot diagrams related by pseudo isotopy, with probability labels transported along the moves and with any pre-crossing introduced by a \(PR1\) move assigned an arbitrary label in \([0,1]\). Then
\[
\mathrm{Spec}(P)=\mathrm{Spec}(P').
\]
\end{proposition}

\begin{proof}
All Reidemeister and pseudo Reidemeister moves except \(PR1\) induce a bijection between complete resolutions of \(P\) and those of \(P'\), obtained by transporting probability labels along the isotopy. Corresponding resolutions preserve crossing signs and produce ambient isotopic classical knots. Since the probabilities assigned to corresponding resolutions coincide, these moves preserve the induced probability distribution on classical knot types.

It remains to consider the move \(PR1\). Suppose that a pre-crossing with probability label \(p\) is introduced or removed by a \(PR1\) move. Its two possible resolutions are a positive kink and a negative kink, both of which are removed by the classical Reidemeister move \(R1\) and therefore yield the same classical knot type. The total probability contribution of these two resolutions is \(p+(1-p)=1\). Thus \(PR1\) does not change the aggregated probability assigned to any classical knot type. Hence the spectrum is invariant under full pseudo isotopy preserving probability labels.
\end{proof}

\begin{definition}\label{def:spectrum-invariant}
Let $I$ be a classical knot invariant. The \emph{spectrum of $I$} with respect to $P$ is the induced probability distribution on invariant values defined by

\[
\mathrm{Spec}_I(P)
=
\left\{
(a,\mathbb P_P[I=a])
\ \middle|\
a\in I(\mathcal K(P)),\;
\mathbb P_P[I=a]>0
\right\}.
\]

where
\[
\mathbb P_P[I=a]
=
\sum_{\substack{K\in\mathcal K(P)\\ I(K)=a}}
\mathbb P_P[K].
\]
\end{definition}

\begin{remark}\rm
The spectrum of an invariant describes the distribution of values attained by $I$ across all classical resolutions of $P$. Statistical quantities such as expectations or variances may be computed for numerical invariants such as crossing number, genus, or signature.
\end{remark}

Thus the spectrum $\mathrm{Spec}(P)$ summarizes the probabilistic topological behaviour of $P$ by encoding the full distribution of classical knot types arising from its resolutions.

%%%%%%%%%%%%%%%%%%%%%%%%%%%%%%%%%%%%%%%%%%%%%%%%%%%%%%%%%%%%%%%%%%%%%%%%%%%%%%%%%%%%%%%%%%%%%%%%%%%%%%%%%%%%%%%%%%%%%%%%%%%%%%%%%%%%%%%%%%%%%%%%%%%%%

\section{Probabilistic Equivalence}\label{sec:probeq}

In classical knot theory, two knots are equivalent if they are related by a finite sequence of Reidemeister moves. This notion of equivalence is purely deterministic, reflecting the fact that classical knot diagrams contain no uncertainty. In probabilistic pseudo knot theory, pre-crossings are equipped with probability assignments, and each diagram determines a probability distribution on classical knot types through its set of resolutions. Consequently, equivalence must take into account not only diagrammatic transformations but also the probabilistic information carried by these assignments. We therefore introduce the notion of \emph{probabilistic equivalence}, which compares probabilistic pseudo knots through the probability distributions induced by their classical resolutions. This notion complements pseudo isotopy by allowing diagrams with different probability decorations to be compared at the level of their induced topological behaviour.

\subsection{Resolution Distributions}

To define probabilistic equivalence, we associate to each probabilistic pseudo knot a probability distribution describing the likelihood of obtaining each classical knot type after resolution.

\begin{definition}\rm
Let $K$ be a probabilistic pseudo knot and let
\[
\mathcal K(K)
:=
\{\text{classical knot types arising among complete resolutions of }K\}.
\]

The {\it resolution distribution} of $K$, denoted $\mu_K$, is the probability measure on $\mathcal K(K)$ defined by
\[
\mu_K(C)
:=
\mathbb P_K[C],
\]
where $\mathbb P_K[C]$ denotes the probability that a complete resolution of $K$ yields a classical knot of type $C$.

In particular,
\[
\sum_{C\in \mathcal K(K)} \mu_K(C)=1.
\]
\end{definition}

To compare two probabilistic pseudo knots $K_1$ and $K_2$, we measure the difference between their resolution distributions using the total variation distance.

\begin{definition}\rm \label{tvd}
Let $\mu_{K_1}$ and $\mu_{K_2}$ be the resolution distributions of probabilistic pseudo knots $K_1$ and $K_2$. Their {\it total variation distance} is defined by
\[
d_{TV}(\mu_{K_1},\mu_{K_2})
=
\frac12
\sum_{C\in \mathcal K(K_1)\cup \mathcal K(K_2)}
\left|
\mu_{K_1}(C)-\mu_{K_2}(C)
\right|,
\]
where probabilities are understood to be zero for knot types not appearing among the resolutions of a given diagram.
\end{definition}

\begin{remark}\rm
The total variation distance satisfies
\[
0\le d_{TV}(\mu_{K_1},\mu_{K_2})\le 1.
\]
Moreover, $d_{TV}=0$ if and only if the resolution distributions coincide, while $d_{TV}=1$ if and only if their probabilistic supports are disjoint.
\end{remark}

%%%%%%%%%%%%%%%%%%%%%%%%%%%%%%%%%%%%%%%%%%%%%%%%%%%%%%%%%%%%%%%%%%%%%%%%

\subsection{Probabilistic Equivalence}

We now introduce a notion of \emph{approximate probabilistic equivalence}, which compares probabilistic pseudo knots through the proximity of their resolution distributions.

\begin{definition}\rm
Fix a tolerance parameter $\epsilon\geq 0$. Two probabilistic pseudo link diagrams (considered modulo label-preserving pseudo isotopy) $K_1$ and $K_2$ are said to be {\it probabilistically equivalent}, denoted
\[
K_1 \sim_{\mathbb P} K_2,
\]
if their resolution distributions satisfy
\[
d_{TV}(\mu_{K_1},\mu_{K_2}) \le \epsilon,
\]
where $d_{TV}$ denotes the total variation distance.
\end{definition}

\noindent Throughout this paper the tolerance $\epsilon\geq 0$ is fixed in advance and regarded as an external similarity parameter.

\begin{remark}\rm
\begin{enumerate}

\item Probabilistic equivalence is a tolerance-based similarity relation rather than a strict equivalence relation. When $\epsilon=0$, it coincides with equality of resolution distributions and therefore defines a genuine equivalence relation, since total variation distance is a metric.

\item For $\epsilon>0$, probabilistic equivalence is generally not transitive. It is possible to have
\[
K_1 \sim_{\mathbb P} K_2
\quad \text{and} \quad
K_2 \sim_{\mathbb P} K_3,
\]
while
\[
K_1 \not\sim_{\mathbb P} K_3.
\]

\item Probabilistic equivalence is not a purely topological equivalence relation. Rather, it provides a metric notion of similarity comparing probabilistic pseudo knots through their stochastic resolution behaviour.

\end{enumerate}
\end{remark}

\begin{ex}\rm
Consider two probabilistic pseudo Hopf links $L_1$ and $L_2$, each with a single pre-crossing between their components. For $L_1$ let $p_1=0.8$, and for $L_2$ let $p_2=0.85$.

The corresponding resolution distributions are

\[
\mu_{L_1}(C_+)=0.8,
\qquad
\mu_{L_1}(C_-)=0.2,
\]

and

\[
\mu_{L_2}(C_+)=0.85,
\qquad
\mu_{L_2}(C_-)=0.15,
\]

where $C_+$ and $C_-$ denote the classical links obtained by resolving the pre-crossing positively or negatively.

Hence

\[
d_{TV}
=
\frac12
\Big(
|0.8-0.85|
+
|0.2-0.15|
\Big)
=
0.05 .
\]

If $\epsilon=0.1$, then $L_1 \sim_{\mathbb P} L_2$.
\end{ex}

%%%%%%%%%%%%%%%%%%%%%%%%%%%%%%%%%%%%%%%%%%%%%%%%%%%%%%%%%%%%%%%%%%%%%%%

\subsubsection{Interpretation of Probabilistic Equivalence}

Probabilistic equivalence compares probabilistic pseudo knots through the proximity of their resolution distributions in total variation distance, rather than through exact equality of classical invariants. Small differences in pre-crossing probabilities, or discrepancies arising from low-probability resolutions, are tolerated whenever the induced probability distributions on knot types remain sufficiently close.

\smallbreak

Situations may arise in which a classical knot type appears among the resolutions of one probabilistic pseudo knot but not the other. Such asymmetries are automatically captured by the total variation distance, which penalizes missing outcomes proportionally to their probability mass. Consequently, mismatches caused by rare or low-probability resolutions have only a limited effect on probabilistic equivalence.

In Example~\ref{exas} we illustrate this behaviour using two probabilistic pseudo knots with different resolution supports.

\begin{ex}\label{exas}\rm
Consider a probabilistic pseudo knot $K_1$ with possible resolved knot types $C_1,C_2,$ and $C_3$ satisfying
\[
\mu_{K_1}(C_1)=0.55,\qquad
\mu_{K_1}(C_2)=0.40,\qquad
\mu_{K_1}(C_3)=0.05.
\]

Let $K_2$ be a probabilistic pseudo knot with resolutions $C_1$ and $C_2$ given by
\[
\mu_{K_2}(C_1)=0.60,
\qquad
\mu_{K_2}(C_2)=0.40.
\]

Then
\[
d_{TV}
=
\frac12
\left(
|0.55-0.60|
+
|0.40-0.40|
+
|0.05-0|
\right)
=
0.05 .
\]

If $\epsilon=0.1$, then $K_1\sim_{\mathbb P}K_2$.

The knot type $C_3$, which is absent from $K_2$, carries only a small probability mass and therefore contributes minimally to the total variation distance. The dominant knot types $C_1$ and $C_2$ have very similar probabilities in both knots, ensuring probabilistic equivalence.
\end{ex}

%%%%%%%%%%%%%%%%%%%%%%%%%%%%%%%%%%%%%%%%%%%%%%%%%%%%%%%%%%%%%%%%%%%%%%%

\subsection{Weighted distances}

Each probabilistic pseudo knot \(K\) determines a resolution distribution
\[
\mu_K:\mathcal K(K)\to[0,1],
\qquad
\sum_{C\in\mathcal K(K)}\mu_K(C)=1,
\]
where \(\mathcal K(K)\) denotes the set of classical knot types arising among the complete resolutions of \(K\).

In many applications, not all resolved knot types carry equal importance. Certain knot types may represent dominant physical configurations or topologically significant phenomena, while others may correspond to rare or negligible outcomes. To incorporate such priorities, we introduce a bounded weight function
\[
w:\mathcal K \to [0,\infty),
\]
assigning to each classical knot type a nonnegative importance weight. Here \(\mathcal K\) denotes the set of all classical knot types.

\begin{definition}\rm
The \emph{weighted total variation discrepancy} between probabilistic pseudo knots \(K_1\) and \(K_2\) is defined by
\[
d_{TV}^{w}(\mu_{K_1},\mu_{K_2})
:=
\frac12
\sum_{C\in\mathcal K(K_1)\cup\mathcal K(K_2)}
w(C)
\big|
\mu_{K_1}(C)-\mu_{K_2}(C)
\big|,
\]
where probabilities are understood to be zero for knot types not appearing among the resolutions of a given diagram.

When \(w\equiv 1\), this reduces to the standard total variation distance.
\end{definition}

\begin{remark}\rm
Unlike the standard total variation distance, \(d_{TV}^{w}\) need not be bounded above by \(1\) unless the weight function is suitably normalized. Thus \(d_{TV}^{w}\) should be regarded as a weighted discrepancy or similarity measure rather than a probability metric in the usual sense.
\end{remark}

\begin{definition}
Fix a bounded weight function \(w\) and a tolerance \(\epsilon>0\). Two probabilistic pseudo knots \(K_1\) and \(K_2\) are said to be \emph{\(w\)-probabilistically equivalent} if
\[
d_{TV}^{w}(\mu_{K_1},\mu_{K_2})\le \epsilon .
\]
\end{definition}

As in the unweighted case, this notion defines a tolerance-based similarity relation rather than a strict equivalence relation whenever \(\epsilon>0\).

In the following examples we illustrate how \(w\)-probabilistic equivalence refines the comparison of probabilistic pseudo knots by emphasizing, or de-emphasizing, discrepancies in selected knot types.

\begin{ex}\rm
Suppose two probabilistic pseudo knots \(K_1\) and \(K_2\) have resolution distributions supported on knot types \(C_1\) and \(C_2\), given by
\[
\mu_{K_1}(C_1)=0.95,\quad \mu_{K_1}(C_2)=0.05,
\]
\[
\mu_{K_2}(C_1)=0.90,\quad \mu_{K_2}(C_2)=0.10.
\]
The standard total variation distance is
\[
d_{TV}(\mu_{K_1},\mu_{K_2})
=
\frac12\big(|0.95-0.90|+|0.05-0.10|\big)
=
0.05.
\]

Suppose now that \(C_2\) represents a particularly important topological phenomenon. Choosing
\[
w(C_1)=1,
\qquad
w(C_2)=10,
\]
yields
\[
d_{TV}^{w}
= \frac12\big(1\cdot|0.95-0.90|
+
10\cdot|0.05-0.10|\big)
= 0.275.
\]

Although the unweighted discrepancy is small, the weighted discrepancy detects the significant difference in the important knot type \(C_2\).
\end{ex}

\begin{ex}\rm
Consider probabilistic pseudo knots \(K_1'\) and \(K_2'\) with resolution distributions supported on knot types \(D_1\) and \(D_2\):
\[
\mu_{K_1'}(D_1)=0.80,\quad \mu_{K_1'}(D_2)=0.20,
\]
\[
\mu_{K_2'}(D_1)=0.78,\quad \mu_{K_2'}(D_2)=0.22.
\]
Then
\[
d_{TV}(\mu_{K_1'},\mu_{K_2'})
=
\frac12(0.02+0.02)
=
0.02 .
\]

If the weight function is constant, for example
\[
w(D_1)=w(D_2)=1,
\]
the weighted and unweighted distances coincide.

More generally, scaling the weights by a constant factor \(\lambda>0\) gives
\[
d_{TV}^{\lambda w} = \lambda\, d_{TV}^{w}.
\]
Therefore the corresponding similarity notion remains unchanged after rescaling the tolerance \(\epsilon\) by the same factor.
\end{ex}

%%%%%%%%%%%%%%%%%%%%%%%%%%%%%%%%%%%%%%%%%%%%%%%%%%%%%%%%%%%%%%%%%%%%%%%%%
%%%%%%%%%%%%%%%%%%%%%%%%%%%%%%%%%%%%%%%%%%%%%%%%%%%%%%%%%%%%%%%%%%%%%%%%%

\section{Probabilistic Distances and Symmetries}\label{distsym}

In this section we investigate two fundamental aspects of the probabilistic framework: distance-based similarity measures and probabilistic symmetries. Both notions are formulated in terms of the resolution distributions introduced in the previous sections.

Probabilistic distances quantify the proximity between probabilistic pseudo knots or links by comparing their induced distributions on classical resolutions. These measures allow one to detect both small probabilistic perturbations and structurally significant differences between diagrams. Probabilistic symmetries extend classical symmetry considerations to the probabilistic setting by studying invariance properties of resolution distributions under transformations such as mutation, reflection, and mirroring. In this way, classical geometric symmetries are examined through their interaction with probabilistic crossing structures.

%%%%%%%%%%%%%%%%%%%%%%%%%%%%%%%%%%%%%%%%%%%%%%%%%%%%%%%%%%%%%%%%%%%%%%%%%

\subsection{Probabilistic Gordian Distance}

The Gordian distance between two classical knots is the minimum number of crossing changes required to transform one knot into the other \cite{Mu}. We introduce a probabilistic analogue of this notion for probabilistic pseudo knots and links, obtained by averaging the classical Gordian distance over their resolution distributions.

\begin{definition}
Let $K_1$ and $K_2$ be probabilistic pseudo knots (or links). Let
\[
R_i := R(K_i)
\]
denote the finite set of classical knot (or link) types appearing among the complete resolutions of $K_i$, and let
$\mu_{K_1}$ and $\mu_{K_2}$ be their associated resolution distributions. Set
\[
R := R_1 \cup R_2,
\]
and extend both distributions to $R$ by defining $\mu_{K_i}(C)=0$ for $C\notin R_i$.

The \emph{probabilistic Gordian distance} between \(K_1\) and \(K_2\) is defined by
\[
W_G(K_1,K_2)
:=
\inf_{\gamma\in\Gamma(\mu_{K_1},\mu_{K_2})}
\sum_{C_1\in R}\sum_{C_2\in R}
\gamma(C_1,C_2)\,
d_G(C_1,C_2),
\]
where \(d_G(C_1,C_2)\) denotes the classical Gordian distance between the classical knot or link types \(C_1\) and \(C_2\), and where \(\Gamma(\mu_{K_1},\mu_{K_2})\) denotes the set of all couplings of the probability measures \(\mu_{K_1}\) and \(\mu_{K_2}\). That is, \(\gamma\) ranges over all probability distributions on \(R\times R\) whose first and second marginals are \(\mu_{K_1}\) and \(\mu_{K_2}\), respectively.
\end{definition}

Equivalently, \(W_G(K_1,K_2)\) is the minimum expected Gordian distance between paired random resolutions of \(K_1\) and \(K_2\), where the pairing is allowed to vary over all couplings of their resolution distributions. Since the resolution supports are finite, the above infimum is taken over a compact polytope and is well defined.

\begin{remark}\rm
One may also consider the \emph{independent expected Gordian discrepancy}
\[
\overline d_G^{\, \mathbb P}(K_1,K_2)
:=
\sum_{C_1\in R}\sum_{C_2\in R}
\mu_{K_1}(C_1)\,
\mu_{K_2}(C_2)\,
d_G(C_1,C_2).
\]
This quantity is the expected Gordian distance between two independent random resolutions
\[
C_1\sim\mu_{K_1},
\qquad
C_2\sim\mu_{K_2}.
\]
It is a useful descriptive statistic, but it is not a metric. In general,
\[
\overline d_G^{\,\mathbb P}(K,K)\neq 0,
\]
because two independent random resolutions of the same probabilistic pseudo knot may yield different classical knot types.
\end{remark}

\begin{ex}\rm
Consider two probabilistic pseudo knots $K_1$ and $K_2$, each with a single pre-crossing. Suppose that the positive and negative resolutions differ by a single crossing change, so that the classical Gordian distances satisfy
\[
d_G(C^{+},C^{+})=0,
\qquad
d_G(C^{-},C^{-})=0,
\qquad
d_G(C^{+},C^{-})=d_G(C^{-},C^{+})=1 .
\]

Let
\[
\mu_{K_1}(C^{+})=0.7,
\qquad
\mu_{K_1}(C^{-})=0.3,
\]
and
\[
\mu_{K_2}(C^{+})=0.6,
\qquad
\mu_{K_2}(C^{-})=0.4 .
\]

Since both distributions are supported on \(C^+\) and \(C^-\), and since
\[
d_G(C^+,C^-)=1,
\]
the optimal coupling pairs the common mass as much as possible. Thus
\[
W_G(K_1,K_2)
=
|\mu_{K_1}(C^+)-\mu_{K_2}(C^+)|
=
|0.7-0.6|
=
0.1.
\]

On the other hand, the independent expected Gordian discrepancy is
\[
\overline d_G^{\,\mathbb P}(K_1,K_2)
=
0.7\cdot0.6\cdot0
+
0.7\cdot0.4\cdot1
+
0.3\cdot0.6\cdot1
+
0.3\cdot0.4\cdot0
=
0.46 .
\]
This illustrates the difference between the optimal-transport distance \(W_G\) \cite{Villani2009} and the independent expected discrepancy \(\overline d_G^{\,\mathbb P}\).
\end{ex}

In Example~\ref{Gdistmism} we evaluate the independent expected Gordian discrepancy between probabilistic pseudo knots with partially disjoint resolution supports.

\begin{ex}\label{Gdistmism}\rm
Let
\[
R(K_1)=\{C_1^+,C_1^-\},
\qquad
\mu_{K_1}(C_1^+)=0.6,
\quad
\mu_{K_1}(C_1^-)=0.4,
\]
and
\[
R(K_2)=\{C_2^+,C_3\},
\qquad
\mu_{K_2}(C_2^+)=0.5,
\quad
\mu_{K_2}(C_3)=0.5,
\]
where $C_3$ does not appear among the resolutions of $K_1$.

Assume
\[
d_G(C_1^+,C_2^+)=1,
\qquad
d_G(C_1^+,C_3)=3,
\qquad
d_G(C_1^-,C_2^+)=2,
\qquad
d_G(C_1^-,C_3)=1 .
\]

Then
\[
\begin{aligned}
d_G^{\, \mathbb P}(K_1,K_2)
&=
0.6\cdot0.5\cdot1
+
0.6\cdot0.5\cdot3
+
0.4\cdot0.5\cdot2
+
0.4\cdot0.5\cdot1
\\
&=
0.3+0.9+0.4+0.2
=
1.8 .
\end{aligned}
\]

The larger value reflects structural differences between the resolution distributions, particularly due to the presence of the additional resolution $C_3$ in $K_2$.
\end{ex}

\begin{remark}\rm
The probabilistic Gordian distance \(W_G\) satisfies:
\begin{itemize}
\item[-] non-negativity:
\[
W_G(K_1,K_2)\ge 0;
\]

\item[-] symmetry:
\[
W_G(K_1,K_2)=W_G(K_2,K_1);
\]

\item[-] normalization:
\[
W_G(K,K)=0.
\]
\end{itemize}
Moreover, if the classical Gordian distance is a metric on the class of knot or link types under consideration, then \(W_G\) defines a metric on resolution distributions.
\end{remark}

%%%%%%%%%%%%%%%%%%%%%%%%%%%%%%%%%%%%%%%%%%%%%%%%%%%%%%%%%%%%%%%%%%%%%%%%%

\subsection{Probabilistic mutation and symmetries}

In classical knot theory, a \emph{mutation} replaces a tangle by a \(180^\circ\) rotation inside a Conway sphere~\cite{Co}. Such operations preserve certain invariants, for example the Jones polynomial, while potentially altering others. In the probabilistic setting, this idea extends naturally to probabilistic pseudo knots by allowing mutations of tangles together with transformations of pre-crossing probabilities. We distinguish two related operations.

\begin{itemize}
\item[-] A \emph{probabilistic mutation} performs a Conway mutation of a tangle while preserving the probabilities assigned to its pre-crossings.

\item[-] A \emph{flipped probabilistic mutation} performs a Conway mutation and simultaneously replaces each pre-crossing probability \(p\) by \(1-p\), modelling a reversal of local crossing preference.
\end{itemize}

\begin{definition}\rm
A \emph{probabilistic mutation} replaces a tangle \(T\) in a probabilistic pseudo knot or link by a Conway mutation of \(T\), with orientations transported along the mutation, while leaving the probabilities \(p_i\) of all pre-crossings unchanged.

A \emph{flipped probabilistic mutation} replaces a tangle \(T\) by a Conway mutation and simultaneously replaces the probability of each pre-crossing \(pc_i\in T\) by \(1-p_i\).
\end{definition}

\begin{example}\rm
Assume independence of pre-crossing resolutions. Let \(K\) contain a tangle \(T\) with two pre-crossings labelled
\[
p_1=0.6,
\qquad
p_2=0.4.
\]
For a local resolution determined by \(\sigma_1,\sigma_2\in\{0,1\}\), where \(\sigma_i=1\) denotes a positive resolution and \(\sigma_i=0\) denotes a negative resolution, the local probability weight is
\[
\mathbb P_K(\sigma_1,\sigma_2)
=
p_1^{\sigma_1}(1-p_1)^{1-\sigma_1}
p_2^{\sigma_2}(1-p_2)^{1-\sigma_2}.
\]

After a flipped probabilistic mutation, the probabilities become
\[
1-p_1=0.4,
\qquad
1-p_2=0.6.
\]
Thus the corresponding local probability weight becomes
\[
\mathbb P'_K(\sigma_1,\sigma_2)
=
(1-p_1)^{\sigma_1}p_1^{1-\sigma_1}
(1-p_2)^{\sigma_2}p_2^{1-\sigma_2}.
\]

For example, when \(\sigma_1=1\) and \(\sigma_2=0\), we obtain
\[
\mathbb P_K(1,0)
=
0.6\cdot 0.6
=
0.36,
\]
whereas
\[
\mathbb P'_K(1,0)
=
0.4\cdot 0.4
=
0.16.
\]
\end{example}

\begin{remark}\rm
Flipped probabilistic mutation redistributes probability mass among local resolution states within the mutated tangle. High-probability local states may become low-probability states and vice versa. Consequently, invariants derived from the resolution distribution, such as the spectrum or the probabilistic chirality index, may change.

If all pre-crossings in the mutated tangle satisfy \(p=\frac12\), then the flip \(p\mapsto 1-p\) does not change the probability weights of local sign assignments. However, the induced spectrum on classical knot types may still change if the Conway mutation changes the knot types arising from the corresponding complete resolutions.
\end{remark}

%%%%%%%%%%%%%%%%%%%%%%%%%%%%%%%
%%%%%%%%%%%%%%%%%%%%%%%%%%%%%%%
\subsection{On Chirality}\label{sec:chirality}

Chirality is a fundamental feature of classical knots, capturing whether a knot is equivalent to its mirror image. Invariants such as the Jones polynomial are often used to distinguish chiral from achiral knots. In the context of probabilistic pseudo knots, chirality becomes richer: uncertainty in crossing resolution can lead a diagram to probabilistically yield both chiral and achiral knots.

A key observation is that the likelihood of chiral outcomes depends not only on the diagram's structure but also on the assigned probabilities of each pre-crossing. For example, a pseudo knot with all probabilities set to \( p_i = 0.5 \) treats over- and under-crossings equally, often yielding symmetric behavior. On the other hand, an asymmetric probability distribution can bias the resolution process toward specific chiral types.

Throughout this subsection, chirality is understood in the chosen oriented category: {\it a classical knot type is called chiral if it is not equivalent to its mirror image in that category}. 

To capture this behavior, we introduce the following probabilistic invariant.

\begin{definition}\rm
Let \( K_p \) be a probabilistic pseudo knot with resolution distribution \( \mu_{K_p} \) supported on the set \( \mathcal{K}(K_p) \) of classical knot types arising among its complete resolutions. The \emph{probabilistic chirality index} is defined by
\[
\chi_p(K_p)
=
\sum_{C \in \mathcal{K}(K_p)}
\mu_{K_p}(C)\,\delta(C),
\]
where \( \delta(C)=1 \) if the classical knot type \( C \) is chiral (up to oriented ambient isotopy), and \( \delta(C)=0 \) if it is achiral.
\end{definition}

\begin{remark}\rm
Equivalently,
\[
\chi_p(K_p)=\mathbb{E}[\delta(C)],
\]
where \( C \) is a random classical resolution drawn according to the resolution distribution \( \mu_{K_p} \). Thus \( \chi_p(K_p) \) measures the probability that a random resolution of \( K_p \) is chiral. The index detects the likelihood of chirality but does not distinguish between opposite handedness.
\end{remark}

\begin{proposition}
For any probabilistic pseudo knot \( K_p \), the index satisfies
\[
0 \leq \chi_p(K_p) \leq 1 .
\]
\end{proposition}

\begin{proof}
Since \( \mu_{K_p} \) is a probability distribution and \( \delta(C)\in\{0,1\} \) for all classical knot types \( C \), the sum defining \( \chi_p(K_p) \) is a convex combination of values in the interval \( [0,1] \).
\end{proof}

\begin{remark}\rm
\begin{itemize}
\item[(i)] The index \( \chi_p(K_p) \) is invariant under pseudo-Reidemeister moves and regular isotopy, since the resolution distribution is preserved and chirality is a classical isotopy invariant.

\smallbreak

\item[(ii)] The probabilistic chirality index is not additive under connected sum. In particular, even if two probabilistic pseudo knots \( K_1 \) and \( K_2 \) have large chirality indices, the connected sum \( K_1 \# K_2 \) may admit achiral resolutions. Thus \( \chi_p \) should be interpreted as a global measure of the likelihood of chiral behavior rather than as an additive invariant.
\end{itemize}
\end{remark}

\begin{ex}[Pseudo Trefoil]\label{pptchind}\rm
Suppose a probabilistic pseudo trefoil has three pre-crossings with probabilities \( p_1 = 0.7 \), \( p_2 = 0.6 \), and \( p_3 = 0.5 \). Assume that its resolution spectrum contains:
\begin{itemize}
\item[] \( T^- \): left-handed trefoil (chiral),
\item[] \( T^+ \): right-handed trefoil (chiral),
\item[] \( U \): the unknot (achiral),
\end{itemize}
with probabilities
\[
\mu_{K_p}(T^-) = 0.21,
\qquad
\mu_{K_p}(T^+) = 0.06,
\qquad
\mu_{K_p}(U) = 0.73 .
\]

Then
\[
\chi_p(K_p)
=
0.21 + 0.06
=
0.27 .
\]

\noindent
\textbf{Interpretation.}
Although chiral outcomes occur, the pseudo knot resolves predominantly into an achiral configuration.
\end{ex}

%%%%%%%%%%%%%%%%%%%%%%%%%%%%%%%%%%%%%%%%%%%%%%%%%%%%%%%%%%%%%%%%%%%%%%%%%%%%%%%%%%%%%%%%%%%%%%%%%%%%%%%%%%%%%%%%%%%%%%%%%%%%%%%%%%%%%%%%%%%%%%%%%%%%%%%%%%%%%%%%%%

\section{Numerical Invariants of Probabilistic Pseudo Knots}\label{numinv}

Numerical invariants provide useful quantitative descriptors of knots and links and often serve as building blocks for more elaborate constructions. In this section we introduce probabilistic analogues of classical numerical quantities associated with knot and link diagrams. Some of these, such as the probabilistic writhe and the probabilistic linking number, can be interpreted as expectation-type extensions of their classical counterparts over the resolution distribution. We also introduce the probabilistic crossing tendency, a diagram-level quantity measuring the average bias of pre-crossings toward positive or negative resolutions.

%%%%%%%%%%%%%%%%%%%%%%%%%%%%%%%%%%%%%%%%%%%%%%%%%%%%%%%%%%%%%%%%%%%%%%%%%%

\subsection{The probabilistic writhe}

In the classical case, the writhe of an oriented knot or link diagram is the sum of the signs of all classical crossings, distinguishing positive and negative contributions (see Figure~\ref{si}). It is invariant under regular isotopy, but not under the Reidemeister move \(R1\).

\begin{figure}[H]
\begin{center}
\includegraphics[width=1.3in]{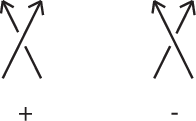}
\end{center}
\caption{The sign of the crossings.}
\label{si}
\end{figure}

In the probabilistic setting, this notion is extended by replacing the sign of a pre-crossing by its expected sign.

\begin{definition}\rm \label{dgw}
Let \(K\) be an oriented probabilistic pseudo knot diagram. Let \(C_{\mathrm{cl}}(K)\) denote the set of classical crossings of \(K\), and let \(C_{\mathrm{pre}}(K)\) denote the set of pre-crossings of \(K\). The \emph{probabilistic writhe} of \(K\) is defined by
\[
pw(K)
:=
\sum_{c\in C_{\mathrm{cl}}(K)} \operatorname{sgn}(c)
+
\sum_{pc_i\in C_{\mathrm{pre}}(K)}(2p_i-1),
\]
where \(p_i\) is the probability that the pre-crossing \(pc_i\) resolves as a positive crossing.
\end{definition}

A pre-crossing with \(p_i=\frac12\), corresponding to equal probabilities for positive and negative resolutions, contributes zero to the probabilistic writhe and is therefore neutral with respect to \(pw\). If all crossings are classical, or equivalently if all probabilistic crossings are degenerate with \(p_i=1\) for positive crossings and \(p_i=0\) for negative crossings, then \(pw(K)\) reduces to the classical writhe
\[
w(K)=\sum_i \operatorname{sgn}(c_i).
\]

\begin{remark}
Equivalently, under the independence assumption for pre-crossing resolutions, the probabilistic writhe is the expected classical writhe over the distribution of complete resolved diagrams:
\[
pw(K)
=
\sum_{r\in\mathcal C(K)}
\mathbb P(r)\,w(r),
\]
where \(\mathcal C(K)\) denotes the set of complete classical resolutions of the diagram \(K\), \(\mathbb P(r)\) is the probability of the complete resolution \(r\), and \(w(r)\) denotes the classical writhe of the resolved diagram \(r\).\end{remark}

\begin{theorem}\label{pwinv}
The probabilistic writhe is invariant under label-preserving regular isotopy of probabilistic pseudo knot diagrams.
\end{theorem}

\begin{proof}
The classical part of the writhe is invariant under the classical Reidemeister moves \(R2\) and \(R3\). Indeed, \(R2\) creates or removes one positive and one negative crossing, whose contributions cancel, while \(R3\) merely rearranges crossings without changing their signs.

For the probabilistic part, label-preserving pseudo Reidemeister moves transport pre-crossings together with their probability labels. Thus the terms \(2p_i-1\) associated with the pre-crossings are preserved. Consequently the total sum defining \(pw(K)\) remains unchanged under label-preserving regular isotopy.
\end{proof}

\begin{theorem}[Stability of probabilistic writhe]
Let \(K_1\) and \(K_2\) be probabilistic pseudo knot diagrams, and let \(\nu_{K_1}\) and \(\nu_{K_2}\) denote their probability distributions on complete resolved diagrams, extended by zero to a common finite support
\[
R:=\mathcal C(K_1)\cup\mathcal C(K_2).
\]
Suppose there exists a constant \(M\ge 0\) such that
\[
|w(r)|\le M
\qquad
\text{for all } r\in R.
\]
If
\[
d_{TV}(\nu_{K_1},\nu_{K_2})\le \varepsilon,
\]
then
\[
|pw(K_1)-pw(K_2)|\le 2M\varepsilon.
\]
\end{theorem}

\begin{proof}
Using the expectation formula,
\[
pw(K)=\sum_{r\in R}\nu_K(r)\,w(r).
\]
Therefore
\[
|pw(K_1)-pw(K_2)|
=
\left|
\sum_{r\in R}
\big(\nu_{K_1}(r)-\nu_{K_2}(r)\big)w(r)
\right|
\le
\sum_{r\in R}
|\nu_{K_1}(r)-\nu_{K_2}(r)|\,|w(r)|.
\]
Since \(|w(r)|\le M\) for all \(r\in R\), we obtain
\[
|pw(K_1)-pw(K_2)|
\le
M\sum_{r\in R}
|\nu_{K_1}(r)-\nu_{K_2}(r)|
=
2M\,d_{TV}(\nu_{K_1},\nu_{K_2})
\le 2M\varepsilon.
\]
\end{proof}

\begin{remark}\rm
The distinction between \(\nu_K\) and \(\mu_K\) is important. The distribution \(\mu_K\) is defined on classical knot types, while \(\nu_K\) is defined on complete resolved diagrams. Since writhe is a regular isotopy invariant of diagrams rather than an ambient isotopy invariant of knot types, the expectation and stability statements for \(pw\) must be formulated using \(\nu_K\), not the knot-type spectrum \(\mu_K\).
\end{remark}

%%%%%%%%%%%%%%%%%%%%%%%%%%%%%%%%%%%%%%%%%%%%%%%%%%%%%%%%%%%%%%%%%%%%%%%%%%

\subsection{The probabilistic linking number}

The linking number in classical knot theory measures how two components of an oriented link are interlinked. For two components \(L_i\) and \(L_j\), it is computed as one half of the sum of the signs of the crossings between these two components. In the probabilistic framework, the linking number must account for uncertainty in crossing signs, while still considering only crossings involving distinct components.

\begin{definition}\rm \label{dpl}
Let \(L\) be an oriented probabilistic pseudo link with components
\[
L_1,L_2,\ldots,L_n.
\]
For \(i\neq j\), let \(C(L_i,L_j)\) denote the set of crossings between strands of \(L_i\) and \(L_j\). Classical crossings are assigned probabilities \(p_k=1\) or \(p_k=0\) according to whether they are positive or negative.

The \emph{probabilistic linking number} between \(L_i\) and \(L_j\) is defined by
\[
pl_{ij}(L)
:=
\frac12
\sum_{c_k\in C(L_i,L_j)}
(2p_k-1),
\]
where \(p_k\) denotes the probability that the crossing \(c_k\) is positive.
\end{definition}

For a two-component link, we simply write \(pl(L)=pl_{12}(L)\). More generally, one may define the total probabilistic linking number by
\[
pl_{\mathrm{tot}}(L)
:=
\sum_{1\le i<j\le n} pl_{ij}(L).
\]

\begin{remark}\rm
\begin{itemize}
\item[(i)] If \(L\) is a classical oriented link, then \(pl_{ij}(L)\) coincides with the classical linking number \(\ell k(L_i,L_j)\).

\item[(ii)] Crossings with \(p_k=\frac12\) do not contribute to \(pl_{ij}(L)\) and are therefore neutral with respect to the probabilistic linking number.
\end{itemize}
\end{remark}

\begin{theorem}\label{plinv}
The probabilistic linking numbers \(pl_{ij}(L)\) are invariant under oriented, label-preserving ambient isotopy of probabilistic pseudo links.
\end{theorem}

\begin{proof}
The classical linking number is invariant under oriented ambient isotopy. It remains to check that the probabilistic contributions behave correctly under the corresponding pseudo-Reidemeister moves.

Moves \(R1\) and \(PR1\) create or eliminate self-crossings within a single component and therefore do not affect the sets \(C(L_i,L_j)\) for \(i\neq j\). The move \(R2\), when it involves two distinct components, creates or removes one positive and one negative crossing between the same pair of components, and their contributions cancel. The move \(R3\) rearranges crossings without changing their signs. The pseudo-Reidemeister moves preserve the labelled pre-crossings and transport their probability labels along the isotopy. Hence the quantities \(pl_{ij}(L)\) are preserved.
\end{proof}

\begin{ex}\rm
Consider a probabilistic pseudo Hopf link \(L\) with two components \(L_1\) and \(L_2\), containing one classical positive crossing and one pre-crossing \(pc_1\) between them, where the pre-crossing has probability \(p_1=0.8\), as illustrated in Figure~\ref{pHopf}.

\begin{figure}[H]
\begin{center}
\includegraphics[width=1.4in]{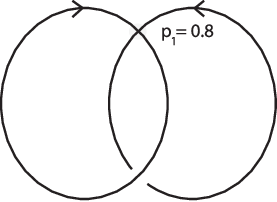}
\end{center}
\caption{A probabilistic pseudo Hopf link.}
\label{pHopf}
\end{figure}

The probabilistic linking number is therefore
\[
pl(L)
=
\frac12\big((2\cdot0.8-1)+(2\cdot1-1)\big)
=
0.8 .
\]

Thus the probabilistic linking number reflects a tendency toward positive linking, consistent with the higher likelihood of a positive crossing.
\end{ex}

Equivalently, under the independence assumption for pre-crossing resolutions, the probabilistic linking number admits an interpretation as the expected classical linking number over the resolution distribution:
\[
pl_{ij}(L)
=
\sum_{C\in R(L)}
\mu_L(C)\,\ell k_{ij}(C),
\]
where \(\ell k_{ij}(C)\) denotes the classical linking number between the \(i\)-th and \(j\)-th components of the resolved oriented link \(C\).

\begin{theorem}[Stability of probabilistic linking number]
Let \(L_1\) and \(L_2\) be probabilistic pseudo links with resolution distributions
\(\mu_{L_1}\) and \(\mu_{L_2}\) supported on the finite set
\[
R:=R(L_1)\cup R(L_2),
\]
extended by zero outside their supports. Let \(\ell k_{ij}(C)\) denote the classical linking number between the corresponding components of a resolved oriented link \(C\in R\).

If
\[
d_{TV}(\mu_{L_1},\mu_{L_2})\le \varepsilon,
\]
then
\[
|pl_{ij}(L_1)-pl_{ij}(L_2)|\le 2M\varepsilon,
\]
where
\[
M:=\max_{C\in R}|\ell k_{ij}(C)|.
\]
\end{theorem}

\begin{proof}
By definition,
\[
pl_{ij}(L)=\sum_{C\in R}\mu_L(C)\,\ell k_{ij}(C).
\]
Hence,
\[
|pl_{ij}(L_1)-pl_{ij}(L_2)|
=
\left|
\sum_{C\in R}
(\mu_{L_1}(C)-\mu_{L_2}(C))\,\ell k_{ij}(C)
\right|
\le
\sum_{C\in R}
|\mu_{L_1}(C)-\mu_{L_2}(C)|\,|\ell k_{ij}(C)|.
\]
Since \(|\ell k_{ij}(C)|\le M\) for all \(C\in R\), we obtain
\[
|pl_{ij}(L_1)-pl_{ij}(L_2)|
\le
M\sum_{C\in R}
|\mu_{L_1}(C)-\mu_{L_2}(C)|
=
2M\,d_{TV}(\mu_{L_1},\mu_{L_2})
\le
2M\varepsilon.
\]
\end{proof}

\begin{remark}\rm
Since the resolution set \(R\) is finite, the constant
\[
M=\max_{C\in R}|\ell k_{ij}(C)|
\]
always exists. The preceding theorem therefore shows that the probabilistic linking number varies at most linearly with the total variation distance between resolution distributions. In other words, \(pl_{ij}\) is Lipschitz continuous with Lipschitz constant \(2M\) on this finite resolution set.
\end{remark}

%%%%%%%%%%%%%%%%%%%%%%%%%%%%%%%%%%%%%%%%%%%%%%%%%%%%%%%%%%%%%%%%%%%%%%%%%%

\subsection{Probabilistic crossing tendency}

The probability labels on the pre-crossings of a probabilistic pseudo knot encode local bias toward positive or negative resolutions. In this subsection we introduce the \emph{probabilistic crossing tendency}, a normalized diagram-level quantity measuring the average inclination of the pre-crossings toward positive resolutions.

\begin{definition}\label{tend}\rm
Let \(K\) be a probabilistic pseudo knot with \(n\geq 1\) pre-crossings labelled
\[
p_1,p_2,\ldots,p_n,
\]
where each \(p_i\in[0,1]\) denotes the probability that the \(i\)-th pre-crossing resolves as positive. The \emph{normalized probabilistic crossing tendency} of \(K\), denoted \(pcp(K)\), is defined by
\[
pcp(K)=\frac{1}{n}\sum_{i=1}^{n}p_i.
\]
\end{definition}

The sum in Definition~\ref{tend} ranges only over pre-crossings. Classical crossings, whose signs are fixed, are omitted so that \(pcp(K)\) measures only the probabilistic part of the diagram.

\begin{remark}\rm
\begin{itemize}
\item[(i)] The probabilistic crossing tendency satisfies
\[
0\le pcp(K)\le 1.
\]

\item[(ii)] The quantity \(pcp(K)\) represents the expected proportion of positive pre-crossings in a random resolution of \(K\). In particular:
\begin{itemize}
\item if \(pcp(K)>\frac12\), the pre-crossings are biased toward positive resolutions;

\item if \(pcp(K)<\frac12\), the pre-crossings are biased toward negative resolutions;

\item if \(pcp(K)=\frac12\), the pre-crossings are balanced on average.
\end{itemize}

Pre-crossings with \(p_i=\frac12\) are neutral with respect to this quantity.
\end{itemize}
\end{remark}

\begin{theorem}
The probabilistic crossing tendency is invariant under label-preserving regular isotopy.
\end{theorem}

\begin{proof}
Under label-preserving regular isotopy, the pre-crossings are transported together with their assigned probabilities. In particular, the pseudo Reidemeister moves \(PR2\) and \(PR3\) preserve the set of labelled pre-crossings and do not alter their probability labels. Consequently, both the sum
\[
\sum_{i=1}^{n}p_i
\]
and the number \(n\) of pre-crossings remain unchanged. Hence \(pcp(K)\) is preserved.
\end{proof}

\begin{ex}\rm
Consider a probabilistic pseudo knot \(K\) with three pre-crossings labelled
\[
p_1=0.8,\qquad p_2=0.6,\qquad p_3=0.4.
\]
Then
\[
pcp(K)=\frac{0.8+0.6+0.4}{3}=0.6,
\]
indicating an average tendency toward positive resolutions among the pre-crossings.
\end{ex}

\begin{theorem}[Probabilistic crossing tendency is not stable]
The probabilistic crossing tendency
\[
pcp(K)=\frac{1}{n}\sum_{i=1}^{n}p_i
\]
is invariant under label-preserving regular isotopy, but it is not stable under probabilistic equivalence defined via total variation distance on resolution distributions. In particular, for any \(\varepsilon>0\), there exist probabilistically equivalent pseudo knots \(K_1\sim_{\mathbb P}K_2\) such that
\[
d_{TV}(\mu_{K_1},\mu_{K_2})<\varepsilon
\quad\text{while}\quad
|pcp(K_1)-pcp(K_2)|=1.
\]
\end{theorem}

\begin{proof}
Fix \(N\ge1\). Let \(K_1\) and \(K_2\) be completely probabilistic pseudo knot diagrams with \(N\) nugatory pre-crossings arranged so that every complete resolution yields the unknot. Then both resolution distributions coincide and are supported entirely on the unknot. Hence
\[
d_{TV}(\mu_{K_1},\mu_{K_2})=0,
\]
and therefore \(K_1\sim_{\mathbb P}K_2\) for every \(\varepsilon>0\).

Choose probabilities so that \(p_i=0\) for all pre-crossings of \(K_1\), and \(p_i=1\) for all pre-crossings of \(K_2\). Then
\[
pcp(K_1)=0
\qquad\text{and}\qquad
pcp(K_2)=1,
\]
and therefore
\[
|pcp(K_1)-pcp(K_2)|=1.
\]
\end{proof}

\begin{remark}\rm
This phenomenon reflects the fact that probabilistic crossing tendency measures local probabilistic bias in a labelled diagram, whereas probabilistic equivalence compares only the induced distributions on resolved knot types. Consequently, \(pcp\) captures diagram-level information that is intentionally invisible to probabilistic equivalence.
\end{remark}

%%%%%%%%%%%%%%%%%%%%%%%%%%%%%%%%%%%%%%%%%%%%%%%%%%%%%%%%%%%%%%%%%%%%%%%%%%%%%%%%%%%%%%%%%%%%%%%%%%%%%%%%%%%%%%%%%%%%%%%%%%%%%%%%%%%%%%%%%%%%%%%%%%%%%%%%%%%%%%%%%%%%%%%%%%%%%%%%%%%%%%%%%%%%%%%%%%%%%%%%%%%%%%%%%%%%%%%%%%%%%%%%%%%%%%%%%%%

\section{Polynomial Invariants}\label{sec:polyinv}

Polynomial invariants constitute some of the most powerful tools in knot theory, providing algebraic encodings of topological and diagrammatic structure. In the classical setting, invariants such as the Alexander polynomial \cite{Alexander}, the Jones polynomial \cite{Jones}, and the Kauffman bracket \cites{Kauffman,LK,LK1} capture important information about knots and links, including chirality, equivalence, and diagrammatic complexity. While the Alexander and Jones polynomials are ambient isotopy invariants of classical knots and links, the Kauffman bracket is naturally a regular isotopy invariant of diagrams and becomes an ambient isotopy invariant only after the usual writhe normalization.

In the probabilistic framework, these constructions must be adapted to accommodate the uncertainty inherent in probabilistic pseudo knots. Pre-crossings carry probabilistic data rather than fixed over--under information, and polynomial constructions may therefore be extended by taking probability-weighted contributions over the possible classical resolutions. The resulting objects encode not only the topology of the resolved diagrams, but also the probability distribution induced by the uncertain crossing data.

In this section we introduce probabilistic polynomial invariants, beginning with a weighted probabilistic Kauffman bracket, followed by expectation-type probabilistic quantum invariants such as Jones-type invariants.

%%%%%%%%%%%%%%%%%%%%%%%%%%%%%%%%%%%%%%%%%%%%%%%%%%%%%%%%%%%%%%%%%%%%%%%%%%

\subsection{The probabilistic bracket polynomial}

The Kauffman bracket is a classical regular isotopy invariant of unoriented knot and link diagrams, defined by skein relations and the loop value. It encodes important diagrammatic information and becomes the Jones polynomial after the usual writhe normalization. For probabilistic pseudo knots, the Kauffman bracket construction extends naturally by incorporating the probabilities assigned to pre-crossings, thereby averaging the brackets of the possible classical resolutions. This construction may be viewed as a probabilistic version of the tangle insertion technique introduced by Henrich and Kauffman~\cite{HenrichKauffman2017} for pseudoknots, singular knots, and rigid vertex spatial graphs.

\begin{definition}\rm
Let \(K\) be a probabilistic pseudo knot diagram with pre-crossings
\[
pc_1,pc_2,\ldots,pc_n,
\]
where each \(pc_i\) is decorated with a probability \(p_i\in[0,1]\). The \emph{weighted probabilistic bracket polynomial} \(\langle K\rangle_{\mathbb P}\) is defined recursively by the following rules.

\begin{itemize}
\item[1.] \textbf{Probabilistic skein relation.}
At a pre-crossing \(pc_i\), we set
\[
\langle K\rangle_{\mathbb P}
=
p_i\langle K_+\rangle_{\mathbb P}
+
(1-p_i)\langle K_-\rangle_{\mathbb P},
\]
where \(K_+\) and \(K_-\) denote the diagrams obtained by resolving \(pc_i\) as a positive and negative classical crossing, respectively.

\item[2.] \textbf{Normalization.}
\[
\langle {\rm O} \rangle_{\mathbb P}=1,
\]
where \({\rm O}\) denotes the unknot.

\item[3.] \textbf{Loop relation.}
\[
\langle D\cup {\rm O}\rangle_{\mathbb P}
=
(-A^2-A^{-2})\langle D\rangle_{\mathbb P}.
\]

\item[4.] \textbf{Classical skein relation.}
For a classical crossing,
\[
\Bigl\langle
\raisebox{-3pt}{\includegraphics[scale=0.6]{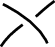}}
\Bigr\rangle_{\mathbb P}
=
A
\Bigl\langle
\raisebox{0pt}{\includegraphics[scale=0.6]{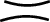}}
\Bigr\rangle_{\mathbb P}
+
A^{-1}
\Bigl\langle
\raisebox{-2pt}{\includegraphics[scale=0.75]{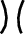}}
\Bigr\rangle_{\mathbb P}.
\]
\end{itemize}

Equivalently,
\[
\langle K\rangle_{\mathbb P}
=
\sum_{r\in\mathcal C(K)}
\mathbb P(r)\,\langle r\rangle,
\]
where \(\mathcal C(K)\) is the set of complete classical resolutions of \(K\), \(\mathbb P(r)\) is the probability of the resolution \(r\), and \(\langle r\rangle\) is the classical Kauffman bracket of the resolved diagram.

If the probabilities \(p_i\) are treated as formal variables, then
\[
\langle K\rangle_{\mathbb P}\in \mathbb Z[A^{\pm1},p_1,\ldots,p_n].
\]
For numerical choices of the probabilities, \(\langle K\rangle_{\mathbb P}\) is the corresponding specialization in \(\mathbb R[A^{\pm1}]\).
\end{definition}

\begin{ex}\label{exkauf}\rm
Consider a probabilistic trefoil diagram with one pre-crossing labelled \(p_1=0.7\). By linearity of the resolution expansion,
\[
\langle K\rangle_{\mathbb P}
=
0.7\,\langle K_+\rangle
+
0.3\,\langle K_-\rangle,
\]
where \(\langle K_+\rangle\) and \(\langle K_-\rangle\) denote the classical Kauffman brackets of the positive and negative resolutions, respectively.
\end{ex}

The weighted probabilistic bracket encodes the resolution probabilities of a probabilistic pseudo knot into a polynomial expression. Each complete resolution contributes its classical Kauffman bracket, weighted according to its probability.

\begin{theorem}\label{wbk}
The weighted probabilistic bracket polynomial is invariant under label-preserving regular isotopy.
\end{theorem}

\begin{proof}
Using the expansion
\[
\langle K\rangle_{\mathbb P}
=
\sum_{r\in\mathcal C(K)}
\mathbb P(r)\,\langle r\rangle,
\]
it suffices to compare complete resolutions before and after a label-preserving regular isotopy move.

The classical Reidemeister moves \(R2\) and \(R3\) preserve the classical Kauffman bracket. The pseudo Reidemeister moves \(PR2\) and \(PR3\), under label-preserving regular isotopy, induce bijections between complete resolutions, preserve the corresponding resolution probabilities, and send corresponding resolved diagrams to diagrams related by classical regular isotopy. Therefore the probabilistic state sums before and after the move agree. Hence \(\langle K\rangle_{\mathbb P}\) is invariant under label-preserving regular isotopy.
\end{proof}

\begin{remark}\rm
The weighted probabilistic bracket is a regular isotopy invariant. As in the classical case, it is not invariant under \(R1\). It is also not expected to be invariant under \(PR1\) unless an additional normalization or convention is imposed.
\end{remark}

%%%%%%%%%%%%%%%%%%%%%%%%%%%%%%%%%%%%%%%%%%%%%%%%%%%%%%%%%%%%%%%%%%%%%%%%%%

\subsubsection{Probabilistic equivalence and the weighted probabilistic bracket}

In this subsection we explore the connection between probabilistic similarity and the weighted probabilistic Kauffman bracket polynomial. Since the Kauffman bracket is a regular isotopy invariant of diagrams, rather than an ambient isotopy invariant of knot types, the comparison in this subsection is made at the level of complete resolved diagrams, or equivalently at the level of regular isotopy classes of complete resolutions.

\smallbreak

\begin{center}
\fbox{\parbox{4.65in}{
{\it If two probabilistic pseudo knots have close diagram-level resolution distributions,
how do their weighted probabilistic Kauffman bracket polynomials differ?}}}
\end{center}

\smallbreak

Let \(\nu_K\) denote the probability distribution on the set \(\mathcal C(K)\) of complete classical resolutions of \(K\). The weighted probabilistic Kauffman bracket is the expectation of the classical Kauffman bracket over these complete resolutions:
\[
\langle K\rangle_{\mathbb P}
=
\sum_{C\in\mathcal C(K)}
\nu_K(C)\,\langle C\rangle,
\]
where \(\langle C\rangle\) denotes the classical Kauffman bracket polynomial of the resolved diagram \(C\).

Thus the probabilistic bracket records how the full range of complete resolved diagrams compatible with \(K\) contributes to its overall bracket polynomial. Probabilistic pseudo knots with similar diagram-level resolution distributions therefore produce similar probabilistic brackets, while significant differences in resolution behaviour lead to corresponding differences in the resulting polynomials.

For a Laurent polynomial
\[
f(A)=\sum_k c_kA^k,
\]
we write
\[
\|f\|_\infty:=\max_k |c_k|
\]
for its coefficient sup norm.

\begin{theorem}\label{thm:bracket_diff_bound}
Let \(K_1\) and \(K_2\) be probabilistic pseudo knots, and let \(\nu_{K_1}\) and \(\nu_{K_2}\) be their probability distributions on complete resolved diagrams, extended by zero to a common finite set
\[
R:=\mathcal C(K_1)\cup \mathcal C(K_2).
\]
If
\[
d_{\mathrm{TV}}(\nu_{K_1},\nu_{K_2})\le \epsilon,
\]
then
\[
\left\|
\langle K_1\rangle_{\mathbb P}
-
\langle K_2\rangle_{\mathbb P}
\right\|_\infty
\le
2\epsilon
\cdot
\max_{C\in R}
\|\langle C\rangle\|_\infty .
\]
\end{theorem}

\begin{proof}
By definition,
\[
\langle K_1\rangle_{\mathbb P}
=
\sum_{C\in R}
\nu_{K_1}(C)\langle C\rangle,
\qquad
\langle K_2\rangle_{\mathbb P}
=
\sum_{C\in R}
\nu_{K_2}(C)\langle C\rangle,
\]
where both distributions have been extended by zero outside their supports. Hence
\[
\langle K_1\rangle_{\mathbb P}
-
\langle K_2\rangle_{\mathbb P}
=
\sum_{C\in R}
\big(\nu_{K_1}(C)-\nu_{K_2}(C)\big)
\langle C\rangle .
\]
Taking the coefficient sup norm gives
\[
\left\|
\langle K_1\rangle_{\mathbb P}
-
\langle K_2\rangle_{\mathbb P}
\right\|_\infty
\le
\sum_{C\in R}
|\nu_{K_1}(C)-\nu_{K_2}(C)|
\,\|\langle C\rangle\|_\infty .
\]
Let
\[
M:=
\max_{C\in R}
\|\langle C\rangle\|_\infty .
\]
Since \(R\) is finite, \(M\) exists. Therefore
\[
\left\|
\langle K_1\rangle_{\mathbb P}
-
\langle K_2\rangle_{\mathbb P}
\right\|_\infty
\le
M
\sum_{C\in R}
|\nu_{K_1}(C)-\nu_{K_2}(C)|.
\]
By the definition of total variation distance,
\[
\sum_{C\in R}
|\nu_{K_1}(C)-\nu_{K_2}(C)|
=
2\,d_{\mathrm{TV}}(\nu_{K_1},\nu_{K_2}).
\]
Thus
\[
\left\|
\langle K_1\rangle_{\mathbb P}
-
\langle K_2\rangle_{\mathbb P}
\right\|_\infty
\le
2M\,d_{\mathrm{TV}}(\nu_{K_1},\nu_{K_2}).
\]
If \(d_{\mathrm{TV}}(\nu_{K_1},\nu_{K_2})\le \epsilon\), the result follows.
\end{proof}

\begin{remark}\rm
The distinction between \(\mu_K\) and \(\nu_K\) is essential. The distribution \(\mu_K\) is defined on classical knot types, while \(\nu_K\) is defined on complete resolved diagrams or their regular isotopy classes. Since the Kauffman bracket is not an ambient isotopy invariant of knot types, stability of the probabilistic bracket must be formulated using \(\nu_K\), not the knot-type spectrum \(\mu_K\).
\end{remark}

%%%%%%%%%%%%%%%%%%%%%%%%%%%%%%%%%%%%%%%%%%%%%%%%%%%%%%%%%%%%%%%%%%%%%%%%%%

\subsection{Probabilistic quantum invariants}

Quantum invariants provide powerful tools for distinguishing and classifying knots and links. In classical knot theory, the term \emph{quantum invariant} usually refers to invariants generalizing the Jones polynomial, often constructed from Markov traces, state-sum models involving solutions of the Yang--Baxter equation, or quantum group methods; see, for example, Kauffman~\cite{KauffmanKnotsPhysics} and Ohtsuki~\cite{OhtsukiQuantumInvariants}. 
Although these invariants are historically and conceptually inspired by quantum theory, they are not necessarily invariants of physical quantum systems in a direct sense.

In the probabilistic setting considered here, such invariants extend naturally by averaging over the classical knot types arising from all complete resolutions of a probabilistic pseudo knot, weighted by their associated resolution probabilities. Thus a probabilistic pseudo knot gives rise to expectation-type quantum invariants, as well as probability distributions of quantum invariant values.

A different, more physical, direction would be to define quantum-probabilistic invariants in a setting related to quantum computation or quantum measurement, where the probabilities arise from quantum states or amplitudes rather than from prescribed classical crossing probabilities. 
This will be the subject of future work.

\begin{definition}\rm
Let \(I\) be a classical knot invariant taking values in a normed vector space \(V\). The \emph{probabilistic quantum invariant} of a probabilistic pseudo knot \(K\) is defined by
\[
I_{\mathbb P}(K)
=
\sum_{C\in\mathcal K(K)}
\mu_K(C)\,I(C),
\]
where \(\mathcal K(K)\) denotes the set of classical knot types arising among the complete resolutions of \(K\), and \(\mu_K(C)\) is the resolution distribution on classical knot types.
\end{definition}

\begin{example}\rm
For the Jones polynomial \(J(C,t)\), the probabilistic version is
\[
J_{\mathbb P}(K,t)
=
\sum_{C\in\mathcal K(K)}
\mu_K(C)\,J(C,t).
\]
\end{example}

\begin{remark}\rm
The probabilistic quantum invariant \(I_{\mathbb P}(K)\) is invariant under label-preserving pseudo isotopy. This follows from the invariance of the classical invariant \(I(C)\), together with the preservation of the resolution distribution \(\mu_K\) under pseudo-Reidemeister moves respecting the probabilistic labels.
\end{remark}

\begin{proposition}[Stability under probabilistic equivalence]
Let \(K_1\) and \(K_2\) be probabilistic pseudo knots with resolution distributions \(\mu_{K_1}\) and \(\mu_{K_2}\). Suppose
\[
d_{\mathrm{TV}}(\mu_{K_1},\mu_{K_2})\le \epsilon.
\]
Then
\[
\|I_{\mathbb P}(K_1)-I_{\mathbb P}(K_2)\|
\le
2\epsilon
\cdot
\max_{C\in\mathcal K(K_1)\cup\mathcal K(K_2)}
\|I(C)\|,
\]
where \(\|\cdot\|\) denotes the chosen norm on \(V\).
\end{proposition}

\begin{proof}
Extend both resolution distributions by zero to the common finite set
\[
R:=\mathcal K(K_1)\cup\mathcal K(K_2).
\]
Then
\[
I_{\mathbb P}(K_1)-I_{\mathbb P}(K_2)
=
\sum_{C\in R}
\big(\mu_{K_1}(C)-\mu_{K_2}(C)\big)I(C).
\]
Taking norms and using the triangle inequality gives
\[
\|I_{\mathbb P}(K_1)-I_{\mathbb P}(K_2)\|
\le
\sum_{C\in R}
|\mu_{K_1}(C)-\mu_{K_2}(C)|\,\|I(C)\|.
\]
Let
\[
M:=
\max_{C\in R}\|I(C)\|.
\]
Since \(R\) is finite,
\[
\|I_{\mathbb P}(K_1)-I_{\mathbb P}(K_2)\|
\le
M\sum_{C\in R}
|\mu_{K_1}(C)-\mu_{K_2}(C)|
=
2M\,d_{\mathrm{TV}}(\mu_{K_1},\mu_{K_2}).
\]
The result follows from the assumption
\(d_{\mathrm{TV}}(\mu_{K_1},\mu_{K_2})\le\epsilon\).
\end{proof}

%%%%%%%%%%%%%%%%%%%%%%%%%%%%%%%%%%%%%%%%%%%%%%%%%%%%%%%%%%%%%%%%%%%%%%%%%%
%%%%%%%%%%%%%%%%%%%%%%%%%%%%%%%%%%%%%%%%%%%%%%%%%%%%%%%%%%%%%%%%%%%%%%%%%%
%%%%%%%%%%%%%%%%%%%%%%%%%%%%%%%%%%%%%%%%%%%%%%%%%%%%%%%%%%%%%%%%%%%%%%%%%%
%%%%%%%%%%%%%%%%%%%%%%%%%%%%%%%%%%%%%%%%%%%%%%%%%%%%%%%%%%%%%%%%%%%%%%%%%%

\section{Probabilistic Seifert Surfaces \& Probabilistic Genus}\label{sec:probSeif}

Seifert surfaces play a central role in classical knot theory, providing two-dimensional surfaces whose boundary is a given knot or link. Classically, one way to construct such a surface is Seifert's algorithm: starting from an oriented knot diagram, one performs oriented smoothings at all crossings, obtains a collection of Seifert cycles, and then attaches half-twisted bands to recover the original diagram.

For probabilistic pseudo knots, the presence of pre-crossings introduces uncertainty. Each pre-crossing may resolve positively or negatively with a specified probability, and each complete resolution determines a classical diagram to which Seifert's algorithm may be applied. Consequently, a probabilistic pseudo knot does not determine a single Seifert surface, but rather a probability distribution of Seifert surfaces arising from its complete resolutions.

In this section, we formalize the notion of \emph{probabilistic Seifert surfaces} as the pushforward of the complete-resolution distribution under the Seifert-surface construction. We then introduce corresponding genus-type quantities, including the minimal genus attainable among the Seifert surfaces arising from complete resolutions and the expected genus with respect to the induced probability distribution.

%%%%%%%%%%%%%%%%%%%%%%%%%%%%%%%%%%%%%%%%%%%%%%%%%%%%%%%%%%%%%%%%%%%%%%%%%%

\subsection{Probabilistic Seifert surfaces}

Let \(K\) be a probabilistic pseudo knot with set of complete resolutions \(\mathcal C(K)\). Let \(\nu_K\) denote the probability distribution on \(\mathcal C(K)\). For each complete classical resolution \(C\in\mathcal C(K)\), let \(S(C)\) denote the Seifert surface obtained from the specific resolved diagram \(C\) via the classical Seifert algorithm.

We now formalize the notion of probabilistic Seifert surfaces.

\begin{definition}\rm
The \emph{probabilistic Seifert surface distribution} of a probabilistic pseudo knot \(K\) is the probability measure \(\mu_S\) on the finite set
\[
\Sigma(K)=\{S(C)\mid C\in\mathcal C(K)\}.
\]
defined by
\[
\mu_S(S)
=
\sum_{\substack{C\in\mathcal C(K):\\ S(C)=S}}
\nu_K(C).
\]
Equivalently, \(\mu_S\) is the pushforward of the complete-resolution distribution \(\nu_K\) under the map
\[
S:\mathcal C(K)\longrightarrow \Sigma(K),
\qquad
C\longmapsto S(C).
\]
It records the probability with which each Seifert surface arises from the complete resolutions of \(K\).
\end{definition}

Since \(\mathcal C(K)\) is finite, the set \(\Sigma(K)\) is finite and \(\mu_S\) defines a finite probability distribution.

\begin{remark}\rm
The equality \(S(C)=S\) should be understood with respect to the chosen level of equivalence for Seifert surfaces. One may work with surfaces produced directly by the Seifert algorithm from diagrams, with ambient isotopy classes of such surfaces, or simply with their genera. The choice should be fixed before applying the distribution \(\mu_S\).
\end{remark}

\smallbreak

To gain geometric intuition, we describe a diagrammatic visualization of how Seifert surfaces vary across resolutions. This is not a new topological construction, but rather a heuristic illustration of how the classical Seifert algorithm behaves under different resolutions.

\smallbreak

\noindent\textbf{Step 1: Resolution and oriented smoothings.}
For a fixed classical resolution of \(K\), apply the classical Seifert algorithm by performing oriented smoothings at all crossings (Figure~\ref{orsm}). The variation between surfaces arises from the different choices of resolutions at the pre-crossings.

\begin{figure}[H]
\begin{center}
\includegraphics[width=1.4in]{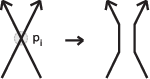}
\end{center}
\caption{Oriented smoothing of a resolved crossing.}
\label{orsm}
\end{figure}

\noindent\textbf{Step 2: Seifert cycles.}
The oriented smoothings produce a collection of disjoint oriented simple closed curves, called Seifert cycles. Each complete resolution yields a corresponding collection of cycles.

\smallbreak

\noindent\textbf{Step 3: Band attachments.}
For each crossing, a half-twisted band is attached between the corresponding Seifert cycles according to the orientation and the resolved crossing sign. In a heuristic visualization, bands arising from former pre-crossings may be labelled by their associated probabilities \(p_i\), reflecting how different resolutions contribute to the overall distribution of surfaces.

\smallbreak

\noindent\textbf{Step 4: Surface assembly.}
Embedding the Seifert cycles and bands in three-dimensional space and thickening the resulting complex produces a Seifert surface corresponding to the chosen resolution. Figure~\ref{pss} illustrates a probabilistic trefoil example.

\begin{figure}[H]
\begin{center}
\includegraphics[width=4.4in]{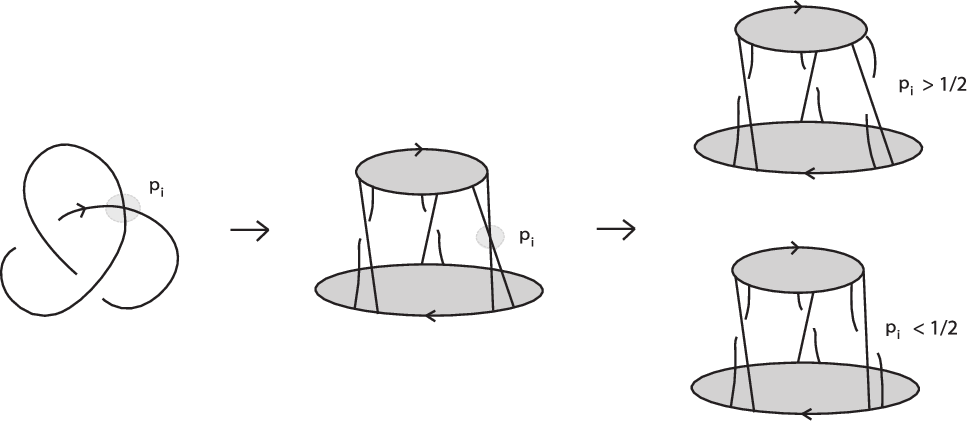}
\end{center}
\caption{Seifert surfaces arising from resolutions of a probabilistic pseudo trefoil. The crossings, or equivalently the corresponding half-twisted bands, are understood to be filled by surfaces, so that the figure represents the associated Seifert surfaces rather than only their boundary curves. Different resolutions contribute distinct classical Seifert surfaces weighted by their probabilities.}

\label{pss}
\end{figure}

\begin{remark}\rm
The procedure above visualizes how Seifert surfaces vary across resolutions, but the probabilistic object is the distribution \(\mu_S\) on classical Seifert surfaces. This distribution captures the uncertainty encoded in the pseudo knot and provides the foundation for probabilistic analogues of genus, Seifert matrices, and related constructions.
\end{remark}

%%%%%%%%%%%%%%%%%%%%%%%%%%%%%%%%%%%%%%%%%%%%%%%%%%%%%%%%%%%%%%%%%%%%%%%%%%

\subsection{Probabilistic genus}

For a classical knot \(C\), the genus \(g(C)\) is the minimal genus among all Seifert surfaces spanning \(C\). For a probabilistic pseudo knot \(K\), the complete resolutions determine a finite set of classical knot types, each with its own classical genus. We first record the smallest genus attainable among these resolved knot types.

\begin{definition}\rm
Let \(K\) be a probabilistic pseudo knot, and let
\[
\mathcal K(K)
\]
denote the set of classical knot types arising among the complete resolutions of \(K\). The \emph{minimal resolution genus}, denoted \(pg(K)\), is defined by
\[
pg(K)=\min\{g(C): C\in\mathcal K(K)\},
\]
where \(g(C)\) denotes the classical knot genus of \(C\).

Thus \(pg(K)\) is the minimal genus among all classical knot types obtained as resolutions of \(K\). In particular, \(pg(K)\) depends only on the set of resolved knot types, and not on their probabilities.
\end{definition}

\begin{remark}\rm
The minimal resolution genus records only the smallest genus that can occur among the classical resolutions of \(K\), and therefore ignores the probabilities assigned to those
resolutions. A complementary probability-sensitive quantity is the \emph{expected resolution genus}, defined by
\[
\mathbb E_g(K)
:=
\sum_{C\in \mathcal K(K)}
\mu_K(C)\,g(C),
\]
where \(\mu_K(C)\) denotes the resolution distribution of \(K\) on classical knot types.
Thus \(\mathbb E_g(K)\) is the expected genus of a random classical resolution of \(K\).
Unlike \(pg(K)\), the quantity \(\mathbb E_g(K)\) depends on the probabilities of the
resolved knot types.
\end{remark}

\begin{theorem}
The minimal resolution genus \(pg(K)\) is invariant under label-preserving pseudo isotopy. In particular, it is invariant under label-preserving regular isotopy.
\end{theorem}

\begin{proof}
All pseudo-Reidemeister moves except \(PR1\) induce bijections between complete resolutions and preserve the corresponding classical knot types. For \(PR1\), the two possible resolutions of the introduced or removed pre-crossing are positive and negative kinks, both of which are removed by the classical Reidemeister move \(R1\) and hence determine the same classical knot type. Thus label-preserving pseudo isotopy preserves the set \(\mathcal K(K)\) of resolved classical knot types.

Since classical knot genus is an ambient isotopy invariant, the genus of each resolved knot type is preserved. Hence the minimum genus across all resolved knot types is preserved.
\end{proof}

\begin{theorem}
The minimal resolution genus is not invariant under probabilistic equivalence.
\end{theorem}

\begin{proof}
Suppose \(K_1\) and \(K_2\) are probabilistic pseudo knots whose resolution distributions are given by:
\[
\mu_{K_1}:
\begin{cases}
0.95 & \text{on a genus \(2\) knot},\\
0.05 & \text{on the unknot},
\end{cases}
\qquad
\mu_{K_2}:
\begin{cases}
0.95 & \text{on the same genus \(2\) knot},\\
0.05 & \text{on a genus \(1\) knot}.
\end{cases}
\]
Then
\[
d_{TV}(\mu_{K_1},\mu_{K_2})
=
\frac{1}{2}
\Big(
|0.95-0.95|
+
|0.05-0|
+
|0-0.05|
\Big)
=
0.05.
\]
Thus \(K_1\) and \(K_2\) are probabilistically equivalent for any tolerance \(\epsilon\ge 0.05\), for example for \(\epsilon=0.1\). However,
\[
pg(K_1)=0
\qquad
\text{while}
\qquad
pg(K_2)=1.
\]
Therefore \(pg\) is not determined by probabilistic equivalence.
\end{proof}

\begin{remark}\rm
The minimal resolution genus depends on the existence of low-genus resolved knot types, not on how likely those knot types are. Consequently, it is stable under pseudo isotopy but not under probabilistic equivalence, which compares resolution distributions up to a tolerance.
\end{remark}

\begin{lemma}[Total variation controls events]
\label{lem:TV-events}
Let \(\mu\) and \(\nu\) be probability measures on a finite set \(\Omega\). If
\[
d_{TV}(\mu,\nu)\le \varepsilon,
\]
then for every subset \(A\subseteq \Omega\),
\[
\nu(A)\ge \mu(A)-\varepsilon.
\]
In particular, if \(\mu(A)>\varepsilon\), then \(\nu(A)>0\).
\end{lemma}

\begin{proof}
Using the characterization
\[
d_{TV}(\mu,\nu)
=
\sup_{A\subseteq\Omega}
|\mu(A)-\nu(A)|,
\]
we have
\[
|\mu(A)-\nu(A)|\le d_{TV}(\mu,\nu)\le\varepsilon
\]
for every \(A\subseteq\Omega\). Hence
\[
\nu(A)\ge \mu(A)-\varepsilon.
\]
\end{proof}

\begin{theorem}
\label{thm:pg-stability-one-sided}
Let \(K_1\) and \(K_2\) be probabilistic pseudo knots with resolution distributions \(\mu_{K_1}\) and \(\mu_{K_2}\) satisfying \(d_{TV}(\mu_{K_1},\mu_{K_2})\le \varepsilon\). Let
\(R:=\mathcal K(K_1)\cup\mathcal K(K_2)\), and define
\[
A:=
\{C\in R: g(C)=pg(K_1)\}.
\]
If \(\mu_{K_1}(A)>\varepsilon\), then \(pg(K_2)\le pg(K_1)\).
\end{theorem}

\begin{proof}
By Lemma~\ref{lem:TV-events}, we have \(\mu_{K_2}(A) \ge \mu_{K_1}(A)-\varepsilon >0\). Hence there exists a resolved knot type \(C\in A\) occurring with positive \(\mu_{K_2}\)-probability. Such a \(C\) therefore appears among the resolutions of \(K_2\), and satisfies \(g(C)=pg(K_1)\). Since \(pg(K_2)\) is the minimum genus among resolved knot types of \(K_2\), it follows that \(pg(K_2)\le pg(K_1)\).
\end{proof}

This shows that the minimal resolution genus is upper semicontinuous in the following sense: it cannot jump upward under a small perturbation of the resolution distribution, provided that the minimal-genus resolved knot types of \(K_1\) carry probability mass greater than the perturbation size. However, since \(pg(K)\) is defined by a minimum over resolved knot types, it is not continuous in a genuinely probabilistic sense. This contrasts with expectation-type quantities, such as the expected resolution genus \(\mathbb E_g(K)\), which vary continuously with the resolution distribution on any fixed finite support.

\subsection{Probabilistic Goeritz matrices and determinants}

Goeritz matrices provide a checkerboard region formulation of classical knot and link invariants and are well suited to probabilistic variants. Given a checkerboard coloring of a probabilistic pseudo link diagram and a fixed Goeritz sign convention, let \(\eta(c)\in{+1,-1}\) denote the classical incidence sign of a classical crossing \(c\). For a pre-crossing \(c_i\) with probability \(p_i\), let \(c_i^+\) and \(c_i^-\) denote its positive and negative resolutions. We define its probabilistic incidence weight by
\[
\eta_{\mathbb P}(c_i)
=
p_i\,\eta(c_i^+)+(1-p_i)\,\eta(c_i^-).
\]
In particular, when the two resolutions have opposite incidence signs, this becomes
\[
\eta_{\mathbb P}(c_i)
=
(2p_i-1)\eta(c_i^+).
\]
For a classical crossing \(c_i\), we set
\(\eta_{\mathbb P}(c_i)=\eta(c_i)\).

Let the relevant checkerboard regions be indexed by \(0,1,\ldots,m\). One may then define a probabilistic Goeritz matrix \(G_{\mathbb P}\) by replacing each classical incidence number in the usual Goeritz construction by its probabilistic counterpart \(\eta_{\mathbb P}(c_i)\). Equivalently, the entries of \(G_{\mathbb P}\) are obtained by summing the probabilistic incidence weights over crossings adjacent to the corresponding regions, using the same sign convention as in the classical Goeritz matrix.

Deleting one row and one column gives a reduced probabilistic Goeritz matrix \(G_{\mathbb P}^{\mathrm{red}}\). This leads to a probabilistic determinant-type quantity
\[
\det_{\mathbb P}(K)
=
\left|
\det\left(G_{\mathbb P}^{\mathrm{red}}\right)
\right|.
\]

This construction should be interpreted with care. Its value may depend on the checkerboard coloring, the Goeritz sign convention, and the chosen diagrammatic representative unless appropriate invariance results are established. Nevertheless, it provides a natural way to deform classical Goeritz-type data by probabilistic crossing information.

\begin{ex}\rm
Consider the checkerboard-colored diagram of a probabilistic figure-eight knot shown in Figure~\ref{fig8}.

\begin{figure}[ht]
\begin{center}
\includegraphics[width=1.8in]{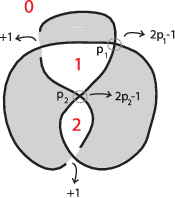}
\end{center}
\caption{A probabilistic figure-eight knot.}
\label{fig8}
\end{figure}

Using a fixed Goeritz sign convention and labeling the white regions by \(0,1,2\), one obtains an unreduced probabilistic Goeritz matrix of the form
\[
G=\left[\begin{matrix}
-1 - 2p_1 & 2p_1       & 1 \\
2p_1     & 1 - 2p_1 - 2p_2 & 2p_2 - 1 \\
1        & 2p_2 - 1   & -2p_2
\end{matrix}\right].
\]
After deleting the prescribed row and column, the reduced matrix produces the determinant-type expression
\[
\det_{\mathbb P}(K)
=
\left|4p_1p_2+2p_2-1\right|.
\]
\end{ex}

These matrix-based constructions illustrate how classical algebraic data may be deformed by probabilistic crossing information. Their main role at this stage is exploratory: they suggest a pathway toward new combinatorial and probabilistic invariants, while leaving their precise invariance properties and dependence on auxiliary choices as important questions for future work.

%%%%%%%%%%%%%%%%%%%%%%%%%%%%%%%%%%%%%%%%%%%%%%%%%%%%%%%%%%%%%%%%%%%%%%%%%%%%%%%%%%%%%%%%%%%%%%%%%%%%%%%%%%%%%%%%%%%%%%%%%%%%%%%%%%%%%%%%%%%%%%%%%%%%%%%%%%%%%%%%%%%%%%%%%%%%%%%%%%%%%%%%%%%%%%%%%%%%%%%%%%%%%%%%%%%%%%%%%%%%%%
%%%%%%%%%%%%%%%%%%%%%%%%%%%

\section{Future Directions}\label{sec:future}

This work introduces a probabilistic framework for knot invariants by assigning probabilities to pre-crossings and studying the resulting probability distributions on classical knot or link types and their invariants. This viewpoint suggests several promising directions for further research, ranging from probabilistic analogues of algebraic invariants to probability distributions on \(3\)-manifolds and extensions of quantum invariants.

\subsection{Further probabilistic matrix-based constructions}

Classical knot invariants such as Seifert matrices, signatures, and Alexander polynomials are derived from combinatorial or algebraic structures associated with Seifert surfaces, checkerboard colorings, or related diagrammatic constructions. In the probabilistic setting, these structures may be extended by considering the distributions of the corresponding classical objects across all complete resolutions. When compatible choices of bases, colorings, or presentations are available, one may further study averaged matrices, expected signatures, determinant distributions, and other derived probabilistic quantities.

\subsubsection*{Probabilistic Seifert matrices and Alexander-type functions}

For each complete classical resolution \(C\in\mathcal C(K)\), one may apply Seifert's algorithm to obtain a Seifert surface \(S(C)\). After choosing a basis for \(H_1(S(C))\), this gives a classical Seifert matrix \(V(C)\).

A first probabilistic object is therefore not necessarily a single matrix, but rather the induced distribution
\[
C\longmapsto V(C)
\]
on Seifert matrices, together with the corresponding probabilities of the resolutions. This avoids the difficulty that different resolutions may produce Seifert surfaces of different genus, and hence Seifert matrices of different sizes.

When the Seifert surfaces arising from the resolutions have compatible first homology groups, and when coherent choices of bases have been fixed, one may further define an averaged Seifert matrix by
\[
V_{\mathbb P}
=
\mathbb E_{C\sim \nu_K}[V(C)]
=
\sum_{C\in\mathcal C(K)}
\nu_K(C)V(C),
\]
where \(\nu_K\) denotes the probability distribution on complete resolutions. Such an averaged matrix depends on the chosen surfaces and bases, and should therefore be regarded as an algebraic statistic associated with the probabilistic diagram rather than as a canonical topological invariant.

By contrast, polynomial invariants obtained from resolved knot types admit a more direct probabilistic averaging. For example, one may define the expected Alexander polynomial by
\[
\Delta_K^{\mathrm{exp}}(x)
=
\sum_{J\in\mathcal K(K)}
\mu_K(J)\Delta_J(x),
\]
where \(\mathcal K(K)\) is the set of classical knot types arising among the resolutions, \(\mu_K\) is the induced resolution distribution on knot types, and \(\Delta_J(x)\) is a chosen normalization of the classical Alexander polynomial of \(J\).

These constructions are algebraic rather than purely topological at the diagrammatic matrix level. They encode how Seifert-type data respond to probabilistic crossing information, but their invariance properties depend on choices of surfaces, bases, and normalizations. A deeper understanding of these dependencies, as well as their computational complexity, remains open for further study.

\subsection{Probabilistic Surgery and Distributions on \(3\)-Manifolds}

Dehn surgery is among the most powerful tools in \(3\)-manifold topology. When applied to a probabilistic pseudo knot, surgery produces not a single \(3\)-manifold, but a probability distribution of \(3\)-manifolds obtained by performing surgery on the classical resolutions.

Fix a surgery coefficient \(r\in\mathbb Q\cup\{\infty\}\). Let \(K\) be a probabilistic pseudo knot with complete-resolution distribution \(\nu_K\). For each complete classical resolution \(C\in\mathcal C(K)\), let
\[
M_r(C)
\]
denote the closed, orientable \(3\)-manifold obtained by \(r\)-surgery on \(C\). This defines a map
\[
M_r:\mathcal C(K)\longrightarrow \mathcal M_3,
\]
where \(\mathcal M_3\) denotes the set of homeomorphism classes of closed, orientable \(3\)-manifolds.

The \emph{probabilistic surgery distribution} associated to \(K\) and \(r\) is the pushforward measure
\[
\mu_{K,r}^{\mathrm{surg}}(M)
=
\sum_{\substack{C\in\mathcal C(K)\\ M_r(C)\cong M}}
\nu_K(C).
\]
Thus the support of the distribution is
\[
\mathcal M_r(K)
=
\{M_r(C): C\in\mathcal C(K)\},
\]
with probabilities obtained by summing over all resolutions producing the same \(3\)-manifold.

This viewpoint aligns naturally with the Lickorish--Wallace theorem: by varying the knot, link, and surgery coefficients, Dehn surgery realizes all closed, orientable \(3\)-manifolds. In the probabilistic setting, unresolved crossing information induces probability distributions on the resulting surgery manifolds. Understanding how resolution uncertainty influences the topology, geometry, and frequency of the resulting manifolds is a rich direction for future work, with potential connections to random \(3\)-manifolds and statistical topology \cite{DunfieldThurston2006}.

\subsection{Quantum invariants under uncertainty}

The probabilistic quantum invariants introduced in Section~\ref{sec:polyinv} provide a first expectation-type approach to quantum-topological information under crossing uncertainty. Several deeper questions remain open beyond this basic construction.

One direction is to study stability more systematically for specific families of invariants, such as the Jones polynomial, the HOMFLYPT polynomial, HOMFLY-PT-type pseudo-link invariants \cite{D1}, and Witten--Reshetikhin--Turaev type invariants, especially under perturbations of the crossing probabilities. Another direction is to compare expectation-type invariants with the full probability distribution of invariant values, as encoded by \(\mathrm{Spec}_I(K)\). Such distributions may retain information lost under averaging.

Further questions concern whether probabilistic quantum invariants admit skein-theoretic or categorical formulations intrinsic to probabilistic pseudo knots, rather than being defined only by averaging over classical resolutions. Such developments could connect probabilistic pseudo knot theory with statistical mechanics, quantum topology, quantum computation, and topological data analysis.

\subsection{Extensions to graphs, bonds, braids, and virtual diagrams}

The probabilistic viewpoint developed in this paper is not limited to pseudo knot diagrams whose only non-classical features are pre-crossings. A natural direction is to extend the framework to diagrams containing several types of local singular or decorated structures simultaneously.

One such extension concerns diagrams with both rigid-vertex nodes and pseudo crossings. Rigid vertices, in the sense of rigid-vertex spatial graph theory, carry a fixed local cyclic ordering of incident edges, while pseudo crossings encode unresolved over--under information. Mixed diagrams containing both types of local data would provide a common setting for studying unresolved crossings together with graph-like singularities.

A related direction, motivated by the theory of bonded knots and links~\cite{DiamantisKauffmanLambropoulou2025}, is to incorporate long bonds and pre-crossings in the same diagram. In such a theory, bonds may connect points on one or more strands, while pre-crossings may occur between ordinary strands, between a strand and a bond, or between bonds. This would lead to probabilistic versions of bonded knots and links, where both the bond structure and the unresolved crossing data contribute to the topology of the object.

Another important direction is the development of braid-theoretic versions of the present framework. One may consider braids with pre-crossings, leading to probabilistic pseudo braids whose closures are probabilistic pseudo links. Such objects naturally raise Markov-type questions and suggest corresponding algebraic descriptions and probabilistic invariants. This braid-theoretic extension is currently being developed by the authors and will be treated in a forthcoming work.

More generally, one may consider embedded graphs whose diagrams contain ordinary crossings, rigid vertices, and pre-crossings. Resolving the pre-crossings then produces a probability distribution on spatial graphs rather than only on knots or links. This suggests probabilistic extensions of graph invariants, rigid-vertex isotopy, and spatial graph polynomials.

The virtual category provides another broad setting for these ideas. Just as one may study pseudo classical knots, one may study pseudo virtual knots by allowing pre-crossings in virtual knot diagrams. One may also allow virtual crossings together with other local nodes, such as rigid vertices or bonded structures. In this setting, the complete resolutions of a pseudo virtual diagram are virtual knots or links, and their behaviour may differ sharply from the classical case. In particular, a virtual pseudo knot need not have the unknot among its possible resolutions; diagrams of Kishino type provide basic examples of virtual phenomena that cannot be detected by simply expecting classical unknot resolutions to
appear.

Finally, the probabilistic framework suggests a computational programme based on knot tables. For a given family of pseudo diagrams, or for classical diagrams in which selected crossings are replaced by pre-crossings, one may compute the resulting distribution of classical knot types across the knot table. This would give knotting probabilities for
tabulated knots and links, allowing one to study how likely different knot types are to arise from unresolved crossing data. Such computations could provide a useful bridge between probabilistic pseudo knot theory, random knot models, and computational knot theory.

\subsection{Potential Applications}\label{sec:applications}

The framework of probabilistic pseudo knots opens up a wide range of potential applications in both theoretical and applied domains.

In molecular biology, these models offer a natural way to describe the topology of DNA, RNA, and proteins \cites{B,WassermanCozzarelli1986,VirnauMirnyKardar2006}, biomolecules whose three-dimensional conformations may include crossings where the over--under information is ambiguous, noisy, or experimentally inaccessible. Imaging and reconstruction techniques, such as cryo-electron microscopy and atomic force microscopy, may produce projected or partially reconstructed data in which some crossing information remains unresolved. Probabilistic invariants provide a principled method for analyzing such uncertainty, enabling the classification of possible folding patterns, the identification of dominant topological features, and the study of molecular dynamics under incomplete information.

In materials science, complex woven, braided, or knotted structures appear in a variety of engineered and natural materials, from textile fabrics and polymer chains \cite{SumnersWhittington1988} to \(3\)-dimensional printed lattices and metamaterials. These structures may be subject to imperfections, deformation, or incomplete observation, resulting in ambiguous or variable topological configurations. The probabilistic invariants introduced here offer new tools for quantifying and comparing such structures despite uncertainty in their exact crossing data.

In computational topology and algorithm design, the probabilistic framework suggests several directions for algorithmic development. Crossing probabilities allow for random sampling of knot resolutions, supporting Monte Carlo simulations, statistical analysis, and robustness testing of topological properties. Such methods may be implemented in computational packages capable of handling probabilistic input and producing probabilistic knot invariants. This could foster new applications across scientific computing, engineering design, and data analysis, particularly in domains where uncertainty is intrinsic.


\begin{thebibliography}{99}

\bibitem{A} {\sc C. C. Adams}, {\em The Knot Book: An Elementary Introduction to the Mathematical Theory of Knots}, 
American Mathematical Society (2004).

\bibitem{Alexander}
{\sc J. W. Alexander},
Topological invariants of knots and links,
{\it Trans. Amer. Math. Soc.} {\bf 30} (1928), 275--306.

\bibitem{BJW} {\sc V. Bardakov, S. Jablan \& H. Wang}, Monoid and group of pseudo braids,  {\em J. Knot Theory and Ramifications}, {\bf 25}, No. 09, 1641002 (2016).

\bibitem{BEHY} {\sc K. Bataineh, M. Elhamdadi, M. Hajij, W. Youmans}, Generating sets of Reidemeister moves of oriented singular links and quandles, {\it J. Knot Theory and Ramifications}, {\bf 27}, No. 14, (2018) 1850064.

\bibitem{B} {\sc D. Buck}, DNA topology, {\em Proceedings of Symposia in Applied Mathematics}, {\bf 66} (2009).

\bibitem{Co} {\sc J. H. Conway}, An enumeration of knots and links, and some of their algebraic properties, {\it Computational Problems in abstract algebra},  (J. Leech,ed.), 329-358. Pergamon Press, 1970.

\bibitem{CrowellFox}
{\sc R. Crowell and R. Fox},
{\em Introduction to Knot Theory},
Springer, 1977.

\bibitem{DiamantisKauffmanLambropoulou2025}
I.~Diamantis, L.~H.~Kauffman, and S.~Lambropoulou,
\newblock Topology and algebra of bonded knots and braids,
\newblock {\em Mathematics} \textbf{13} (2025), no.~20, 3260.


\bibitem{D0} {\sc I. Diamantis}, Algebraic and Geometric Aspects of Non-Classical Knots, arxiv:2607.04445 [math.GT], 2026.

\bibitem{D} {\sc I. Diamantis}, Tied pseudo links \& pseudo knotoids, {\it Mediterr. J. Math.}, {\bf 18}, 201 (2021). https://doi.org/10.1007/s00009-021-01842-1.

\bibitem{D1} {\sc I. Diamantis}, A HOMFLYPT-type invariant for pseudo links via a resolution in Hecke algebras, arXiv:2605.01026 [math.GT], 2026.

\bibitem{DLM3} {\sc I. Diamantis, S. Lambropoulou, S. Mahmoudi}, From annular to toroidal pseudo knots, {\em Symmetry} {\bf 2024}, {\it 16}(10), 1360. https://doi.org/10.3390/sym16101360.

\bibitem{DLM4} {\sc I. Diamantis, S. Lambropoulou, S. Mahmoudi}, From planar to annular to toroidal bracket polynomials for pseudo knots and links, \textit{arXiv:2501.00736 [math.GT]} (2025).

\bibitem{DunfieldThurston2006} {\sc N. M. Dunfield and W. P. Thurston}, Finite covers of random 3-manifolds, \emph{Inventiones Mathematicae} \textbf{166} (2006), 457--521.

\bibitem{EZHLN}
{\sc C. Even-Zohar, J. Hass, N. Linial and T. Nowik},
Invariants of random knots and links,
{\it Discrete Comput. Geom.} {\bf 56} (2016), 274--314.

\bibitem{F} {\sc E. Flapan}, {\em When Topology Meets Chemistry: A Topological Look at Molecular Chirality}, 
Cambridge University Press, Cambridge (2000).

\bibitem{Goeritz}
{\sc L. Goeritz},
Knoten und quadratische Formen,
{\it Math. Z.} {\bf 36} (1933), 647--654.

\bibitem{GL}
{\sc C. McA. Gordon and R. A. Litherland},
On the signature of a link,
{\it Invent. Math.} {\bf 47} (1978), 53--70.

\bibitem{H} {\sc R. Hanaki}, Pseudo diagrams of links, links and spatial graphs, {\it Osaka J. Math.}, {\bf 47}, (2010) 863--883.

\bibitem{HD} {\sc H. A. Dye}, Pseudo knots and an obstruction to cosmetic crossings, {\it J. Knot Theory and Ramifications}, {\bf 26}, No. 04, (2017) 1750022.

\bibitem{HJMR} {\sc A. Henrich, R. Hoberg, S. Jablan, L. Johnson, E. Minten, L. Radovic}, The theory of pseudoknots, {\it J. Knot Theory and Ramifications}, {\bf 22}, No. 07, (2013) 1350032.

\bibitem{HenrichKauffman2017}
A.~Henrich and L.H.~Kauffman,
\newblock Tangle insertion invariants for pseudoknots, singular knots, and rigid vertex spatial graphs,
\newblock {\em Knots, Links, Spatial Graphs, and Algebraic Invariants} (E. Flapan et al., eds.), 
\newblock Contemporary Mathematics, vol. 689, Amer. Math. Soc., 2017, 61--83.

\bibitem{HOMFLY}
{\sc P. Freyd, D. Yetter, J. Hoste, W. B. R. Lickorish, K. Millett and A. Ocneanu},
A new polynomial invariant of knots and links,
{\it Bull. Amer. Math. Soc.} {\bf 12} (1985), 239--246.

\bibitem{Jones}
{\sc V. F. R. Jones},
A polynomial invariant for knots via von Neumann algebras,
{\it Bull. Amer. Math. Soc.} {\bf 12} (1985), 103--111.

\bibitem{Kauffman} {\sc L. H. Kauffman}, State models and the Jones polynomial, \textit{Topology} \textbf{26} (1987), 395--407.

\bibitem{Kauffman1989}
L.~H.~Kauffman,
\newblock Invariants of graphs in three-space,
\newblock {\em Trans. Amer. Math. Soc.} 1989, \textbf{311}, 697--710.

\bibitem{LK} {\sc L. H. Kauffman}, New invariants in the theory of knots, \emph{Amer. Math. Monthly} {\bf 95}, (3) (1988) 195--242.

\bibitem{LK1} {\sc L. H. Kauffman}, Formal Knot Theory, {\it Princeton University Press}, Lecture Notes Series 30 (1983).

\bibitem{KauffmanKnotsPhysics}
L.H.~Kauffman,
\newblock {\em Knots and Physics},
\newblock Series on Knots and Everything, vol.~1,
\newblock World Scientific, Singapore, 1991.

\bibitem{LK2} {\sc L. H. Kauffman}, Introduction to virtual knot theory, {\it J. Knot Theory and Ramifications}, {\bf 21}, No. 13, 1240007 (2012).

\bibitem{Lickorish}
{\sc W. B. R. Lickorish},
A representation of orientable combinatorial 3-manifolds,
{\it Ann. Math.} {\bf 76} (1962), 531--540.

\bibitem{Mu} {\sc H. Murakami}, Some metrics on classical knots, {\em Mathematische Annalen}, {\bf 270} (1985), 35--45.


\bibitem{NOS} {\sc S. Nelson, N. Oyamaguchi, R. Sandanovic}, Psyquandles, Singular Knots and Pseudoknots, {\it Tokyo J. Math.}, {\bf 42}, No. 2, (2019).


\bibitem{OhtsukiQuantumInvariants}
T.~Ohtsuki,
\newblock {\em Quantum Invariants: A Study of Knots, 3-Manifolds, and Their Sets},
\newblock Series on Knots and Everything, vol.~29,
\newblock World Scientific, Singapore, 2002.

\bibitem{RT}
{\sc N. Reshetikhin and V. Turaev},
Ribbon graphs and their invariants derived from quantum groups,
{\it Comm. Math. Phys.} {\bf 127} (1990), 1--26.



\bibitem{Rolfsen}
{\sc D. Rolfsen},
{\em Knots and Links},
AMS Chelsea Publishing, 2003.


\bibitem{Ruberman}
{\sc D. Ruberman},
Mutation and volumes of knots in $S^3$,
{\it Invent. Math.} {\bf 90} (1987), 189--215.


\bibitem{Seifert}
{\sc H. Seifert},
Uber das Geschlecht von Knoten,
{\it Math. Ann.} {\bf 110} (1935), 571--592.

\bibitem{SumnersWhittington1988} {\sc D. W. Sumners and S. G. Whittington}, Knots in self-avoiding walks, \emph{Journal of Physics A: Mathematical and General} \textbf{21} (1988), 1689--1694.

\bibitem{Villani2009} {\sc C. Villani}, Optimal Transport: Old and New, {\it Grundlehren der mathematischen Wissenschaften}, Vol. 338, Springer, Berlin, 2009.

\bibitem{VirnauMirnyKardar2006} {\sc P. Virnau, L. A. Mirny and M. Kardar}, Intricate knots in proteins: Function and evolution, \emph{PLOS Computational Biology} \textbf{2}(9) (2006), e122.

\bibitem{Wallace}
{\sc A. H. Wallace},
Modifications and cobounding manifolds,
{\it Canad. J. Math.} {\bf 12} (1960), 503--528.

\bibitem{WassermanCozzarelli1986} {\sc  S. A. Wasserman and N. R. Cozzarelli}, Biochemical topology: applications to DNA recombination and replication, \emph{Science} \textbf{232} (1986), 951--960.

\bibitem{WRT}
{\sc E. Witten},
Quantum field theory and the Jones polynomial,
{\it Comm. Math. Phys.} {\bf 121} (1989), 351--399.


\end{thebibliography}
\end{document}